\documentclass[11pt,reqno,twoside]{amsart}
\usepackage{tikz}
\usepackage{amssymb}
\usepackage{amsmath}
\usepackage{amsthm}
\usepackage{amsfonts}
\usepackage{amsthm,amscd}
\usepackage{amsbsy}
\usepackage{bm}
\usepackage{url} 
\usepackage{setspace}
\usepackage{float}
\usepackage{psfrag}
\usepackage{tikz}
\usetikzlibrary{calc,intersections}
\usepackage{float}
\usepackage{epstopdf}
\usepackage{caption}
\usepackage[colorinlistoftodos, textsize=scriptsize]{todonotes}
 \usepackage{caption} 
\allowdisplaybreaks

\usetikzlibrary{positioning}

\usepackage[colorlinks, linkcolor=blue]{hyperref}

\restylefloat{figure}

\newtheorem{Theorem}{Theorem}[section]
\newtheorem{Lemma}[Theorem]{Lemma}

\newtheorem{Remark}[Theorem]{Remark}

\newtheorem{Definition}[Theorem]{Definition}

\numberwithin{equation}{section}
\usepackage{setspace}
\usepackage{anysize}

\usepackage[margin=1in]{geometry}
\usepackage[noadjust]{cite}

\def\be{\begin{equation}}
	\def\ee{\end{equation}}
\def\ben{\begin{eqnarray}}
	\def\een{\end{eqnarray}}

\hypersetup{colorlinks=true,%
	citecolor=red,%
	filecolor=blue,%
	linkcolor=blue,%
}

\newcommand{\ncom}{\newcommand}
\usepackage{mathtools}
\mathtoolsset{showonlyrefs=true}

\ncom{\n}{\normalfont}
\ncom{\vp}{\varphi}
\ncom{\pp}{\Phi}
\ncom{\N}{\mathbb{N}}
\ncom{\Lc}{\mathcal}
\ncom{\wt}{\widetilde}
\ncom{\Wf}{\boldsymbol{w}}
\ncom{\Af}{\boldsymbol{A}}
\ncom{\Bf}{\boldsymbol{B}}
\ncom{\Hf}{\boldsymbol{H}}
\ncom{\Pf}{\boldsymbol{P}}
\ncom{\no}{\nonumber}
\ncom{\f}{\boldsymbol{f}}
\ncom{\ub}{u}
\ncom{\y}{y}
\ncom{\vb}{\boldsymbol{v}}
\ncom{\T}{\mathbb{T}}
\ncom{\C}{\mathbb{C}} 
\ncom{\Hb}{\mathbb{H}}
\ncom{\Lb}{\mathbb{L}}
\ncom{\V}{\mathbb{V}}
\ncom{\U}{\mathbb{U}}
\ncom{\Ac}{\mathcal{A}}
\ncom{\Bc}{\mathcal{B}}
\ncom{\af}{\boldsymbol{a}}
\ncom{\pf}{\boldsymbol{p}}
\ncom{\zb}{z}
\ncom{\x}{w}
\ncom{\Y}{\boldsymbol{Y}}
\ncom{\xx}{\boldsymbol{\xi}}
\ncom{\zz}{\boldsymbol{\zeta}}
\ncom{\fb}{\mathbf{f}}
\ncom{\tx}{\tilde{w}}
\ncom{\sz}{\tilde{z}}
\ncom{\tu}{\tilde{u}}
\ncom{\D}{\mathcal{D}}
\newcommand{\vertiii}[1]{{\left\vert\kern-0.25ex\left\vert\kern-0.25ex\left\vert #1 
		\right\vert\kern-0.25ex\right\vert\kern-0.25ex\right\vert}}

	\usepackage{xpatch}
	\makeatletter   
	\xpatchcmd{\@tocline}
	{\hfil\hbox to\@pnumwidth{\@tocpagenum{#7}}\par}
	{\ifnum#1<0\hfill\else\dotfill\fi\hbox to\@pnumwidth{\@tocpagenum{#7}}\par}
	{}{}
	\makeatother    
	
	\makeatletter
	\def\l@subsection{\@tocline{2}{0pt}{4pc}{6pc}{}}
	\def\l@subsubsection{\@tocline{3}{0pt}{8pc}{8pc}{}}
	\makeatother
	
	\makeatletter
	\def\l@section{\@tocline{1}{12pt}{0pt}{}{\bfseries}}
	\makeatother
	\usepackage{ragged2e}
	\justifying

\long\def\/*#1*/{}
\title[Boundary Stabilizability of Generalized Burgers-Huxley Equation]{Boundary Stabilizability of Generalized Burgers-Huxley Equation with Memory}
\date{\today}

\author[M. Bag, W. Akram, and M. T. Mohan]{{Manika Bag$^\ddag$, Wasim Akram$^\dag$, }
	\and{Manil T. Mohan$^\dag$}}
    \thanks{$\dag$ Department of Mathematics, Indian Institute of Technology Roorkee, Uttarakhand, 247667, India, Email- {\normalfont{ wakram2k11@gmail.com; maniltmohan@ma.iitr.ac.in}}\\
$\ddag$ Department of Mathematics, Indian Institute of Science, Education and Research, Thiruvananthapuram,  Kerala, 695551, India, Email- {\normalfont {manikabag19@iisertvm.ac.in}, \normalfont {manikbag058@gmail.com}}\\
Wasim Akram is supported by NBHM (National Board of Higher Mathematics, Department of Atomic Energy) postdoctoral fellowship, No. 0204/16(1)(2)/2024/R\&D-II/10823. This work was initiated during the first author's visit to IIT Roorkee. The first author gratefully acknowledges the project CRG/2021/008278 for providing travel support and the financial support provided by IISER Thiruvananthapuram.}
\makeatletter
\begin{document}
	
	\pagenumbering{arabic}

	\begin{abstract}
In this paper, we study a generalized Burgers-Huxley equation with memory, subject to nonhomogeneous Dirichlet boundary conditions. We construct a linear, finite-dimensional Dirichlet boundary feedback controller aimed at stabilizing the stationary solution corresponding to the homogeneous boundary condition. This controller is designed using eigenfunctions of the Laplace operator. We begin by analyzing the stabilization of a linear system under the proposed feedback law. Subsequently, we demonstrate that the same controller also stabilizes the full nonlinear system by applying the Banach fixed point theorem. Finally, we provide a remark on the stabilization of the generalized Burgers-Huxley equation with memory around the zero solution under nonhomogeneous Dirichlet boundary conditions.

\end{abstract}
\maketitle

\noindent \textbf{MSC Classification (2020).} 93D23 $\cdot$ 93D15 $\cdot$ 35Q35 $\cdot$  35R09


\section{Introduction} 
\subsection{Model problem} Let $\Omega$ be a bounded domain in $\mathbb{R}^d$, $d\in \{2,3\},$ with smooth boundary $\Gamma$. Set $Q=\Omega \times(0,+\infty),$ $\Sigma_1=\Gamma_1\times(0,+\infty)$ and $\Sigma_2=\Gamma_2\times(0,+\infty),$ where $\Gamma=\Gamma_1\cup\Gamma_2.$ This paper aims to investigate the Dirichlet boundary control problem for the following integro-partial differential evolution equation:
\begin{equation}\label{eqn:modelled}
\left\{
\begin{aligned}
& y_t - \eta \Delta y +  a y^\kappa \sum_{i=1}^d \frac{\partial y}{\partial x_i} -  \int_0^t e^{-\delta (t-s)}\Delta y(s)ds  = \beta y(1-y^\kappa)(y^\kappa-\gamma)+ f_s(x), \  \text{ in } \ Q ,\\
&   y(x,t)=u(x,t), \ \text{ for all }  \ (x,t)\in \Sigma_1 , \quad  \frac{\partial y}{\partial n}= 0, \   \text{ on }\  \Sigma_2,\\&y(x,0)=y_0(x), \ \text{ for all } \ x\in \Omega.
\end{aligned}
\right.
\end{equation}
Here, $y$ denotes the state variable and $u$ represents the control variable, which is applied on a portion of the boundary, $\Gamma_1$. The system described above is referred to as the generalized Burgers-Huxley (GBH) equation with memory (see \cite{ABM25GBHE}). The memory effects are incorporated through a linear time convolution involving $\Delta y$ and the memory kernel $e^{-\delta\cdot}$. The Dirichlet controller $u$ is implemented on the boundary segment $\Gamma_1$, while $\Gamma_2$ is assumed to be insulated. The parameters $\eta$, $\delta$, $a$, and $\beta$ are fixed positive constants. The constant $\kappa$ can assume values from a specified set, as described below.
\begin{align} \label{eqdef:delta}
	\kappa=\begin{cases}
		 \text{ any natural number, } & \text{ if } d=2, \\
		 1,2, & \text{ if } d=3.
	\end{cases}
\end{align} 
Note that 
\begin{align*}
    \lim_{t\to \infty} \int_0^t e^{-\delta (t-s)}\Delta \y(s) ds = \lim_{t\to \infty}  \frac{\int_0^t e^{\delta s}\Delta \y(s) ds}{e^{\delta t}} =  \lim_{t\to \infty} \frac{e^{\delta t} \Delta \y(t)}{\delta e^{\delta t}} = \frac{1}{\delta} \Delta \y_\infty.
\end{align*}
Now, we consider the corresponding steady-state equation as:
\begin{equation} \label{eq:StdNSE}
\left\{
    \begin{aligned}
&  - \left(\eta \Delta + \frac{1}{\delta}\right) y_\infty +  a y_\infty^\kappa \sum_{i=1}^d \frac{\partial y_\infty}{\partial x_i}   + \beta y_\infty(y_\infty^\kappa-1)(y_\infty^\kappa-\gamma)  =f_s \ \text{ in } \ \Omega, \\
& y_\infty=0 \  \text{ on }\  \Gamma_1, \quad \frac{\partial y_\infty}{\partial n} = 0,\  \text{ on } \ \Gamma_2.
    \end{aligned}\right.
\end{equation}
This paper focuses on investigating the stabilizability of the steady-state solution $y_\infty$ for the nonlinear system \eqref{eqn:modelled}. The system's instability can arise from the nonlinear terms and the presence of the source term $f_s$.

\begin{Definition}
The system \eqref{eqn:modelled} is said to be \emph{exponentially stabilizable} around the steady-state solution $y_\infty$ in a Banach space $X$ (the state space) with decay rate $\gamma > 0$ if there exists a control $u \in U$ (the control space) such that
\begin{align*}
\|y(t) - y_\infty\|_X \leq C e^{-\gamma t} \|y_0 - y_\infty\|_X,
\end{align*}
for all $t > 0$, where $C > 0$ is a constant independent of both $t$ and the initial condition $y_0$.

\end{Definition}
\subsection{State of the art} If we set $\kappa = 1$ and omit the memory term, equation \eqref{eqn:modelled} reduces to the classical Burgers-Huxley equation, which serves as a prototype model for capturing the interaction between reaction dynamics, convection, and diffusion processes (see \cite{AFK86}). When $\kappa \geq 1$, equation \eqref{eqn:modelled} becomes a generalized form of the Huxley equation, incorporating memory effects, and is relevant in modeling phenomena such as nerve signal propagation in neurophysics and wave propagation in liquid crystals. For a detailed discussion of this system, we refer the reader to \cite{W88}. Solitary and traveling wave solutions of the GBH equation were derived in \cite{WZL90, A05} by employing suitable nonlinear transformations. The global solvability of the system \eqref{eqn:modelled} is investigated in \cite{SAM24}.

Extensive studies have been carried out on the feedback stabilizability of Burgers' and Navier-Stokes equations; see \cite{WKR, MTM-JTP, MTMAnkit, BreKunPf19, SChowErve, JPR6-2017, JPR15SICON, BBSW15, BarbSICON12, JPRSheThev, BarbSICON11, BadESAIM9, AKR24}, among others. For the present discussion, we focus on several notable contributions. In \cite{T80}, R. Triggiani established the existence of stabilizing linear boundary feedback controllers for linear parabolic equations under Dirichlet or Neumann boundary conditions. This foundational approach was later extended to the Navier-Stokes equations in \cite{BLT06, BLTi06}, resulting in robust linear feedback laws. However, these methods typically require solving high- or infinite-dimensional Riccati equations, which pose considerable computational challenges. To address this, \cite{BV13} proposed a novel method for constructing boundary feedback controllers for parabolic systems, under the condition that the normal derivatives of eigenfunctions corresponding to unstable modes are linearly independent. This technique was subsequently applied to the heat equation with fading memory in \cite{MH15}, and extended to more complex systems such as the Navier-Stokes equations with fading memory and the phase-field models in \cite{MNS15, M14}, respectively. The framework was further generalized in \cite{M17} to accommodate semilinear parabolic equations, successfully eliminating the restrictive linear independence assumption required in \cite{BV13}.

\subsection{Contribution and methodology} The main objective of this work is to design a boundary feedback control $u$ such that the controlled solution of equation \eqref{eqn:modelled} satisfies
\begin{align*}
\lim_{t\to\infty} y(t) = y_\infty \quad \text{in } L^2(\Omega),
\end{align*}
with exponential rate of convergence, provided that the initial condition $y_0$ lies within a suitable neighborhood of the steady state $y_\infty$. Specifically, the control is applied on a subset $\Gamma_1$ of the boundary $\Gamma$ to drive the system toward the desired long-term behavior. The goal is to ensure that the solution $y(t)$ converges exponentially to the equilibrium $y_\infty$ as time progresses. This result generalizes the work in \cite{MH15}, which considered a nonlinearity of the form $f(y)$, satisfying
\begin{align*}
f \in C^2(\mathbb{R}), \quad \text{with } |f''(y)| \leq C_1 |y|^\alpha + C_2, \quad y \in \mathbb{R},
\end{align*}
for some constants $C_1, C_2 > 0$ and $\alpha \in \mathbb{N}$, under the assumptions $d = 1, 2$ or $d = 3$ with $0 < \alpha \leq 1$. In contrast, our setting involves a more general nonlinearity $f(y, \nabla y)$ that includes both a Burgers-type convective term and a polynomial-type nonlinear term. An important consequence of the stabilization result is the existence of local  solutions  to the nonlinear system \eqref{eqn:modelled}, a topic that has not been previously addressed in the literature.

To investigate the stabilization of the solution of \eqref{eqn:modelled} around the steady state $y_\infty$, which satisfies \eqref{eq:StdNSE}, we introduce the fluctuation variable $w = y - y_\infty$. This transformation reformulates the original stabilization problem into an equivalent boundary stabilization problem around zero for the following system:
\begin{equation}  \label{eqn:nlmodel}
\left\{
\begin{aligned}
 &  w_t -\eta \Delta w  + a \left( (w+y_\infty)^\kappa \nabla (w+y_\infty) \cdot \boldsymbol{1} - y_\infty^\kappa \nabla y_\infty\cdot \boldsymbol{1} \right)  -  \int_0^t e^{-\delta(t-s)}\Delta w(s) ds \\
		& \quad  + \beta \Big( (w+y_\infty) ((w+y_\infty)^\kappa -1) ((w+y_\infty)^\kappa -\gamma) - y_\infty (y_\infty^\kappa -1) (y_\infty^\kappa -\gamma)  \Big)\\&\quad + \frac{1}{\delta}e^{-\delta t}\Delta y_\infty=0, \ \text{ in }\ \Omega,\\
&   w(x,t)=u(x,t) \ \text{ for all } \ (x,t)\in \Sigma_1 , \quad  \frac{\partial w}{\partial n}= 0, \   \text{ on }\  \Sigma_2,\\&w(x,0)=y_0(x)-y_\infty(x)\  \text{ for all } \  x\in \Omega,
\end{aligned}
\right.
\end{equation}
here $\boldsymbol{1}=(1,1,\ldots,1)^{\top} \in \mathbb{R}^d.$ Note that in the above derivation, we have used
\begin{align*}
    \frac{1}{\delta} \Delta y_\infty =  
    \frac{1}{\delta} (1-e^{-\delta t}) \Delta y_\infty + \frac{1}{\delta} e^{-\delta t} \Delta y_\infty   =   \int_0^t  e^{-\delta (t-s)} \Delta y_\infty ds+ \frac{1}{\delta} e^{-\delta t} \Delta y_\infty.
\end{align*}
In the seminal works \cite{IM15, BV13, M14}, the stabilization of systems was examined either around zero or a non-zero steady state, with the nonlinearity exhibiting polynomial growth. However, in our current setting, the system features both a polynomial-type nonlinearity and an additional coupling nonlinearity involving terms like $y^\kappa$ and $\nabla y$, which introduces further analytical challenges, particularly in studying the stabilization of the linearized problem around a non-constant steady state. To address these challenges, we analyze a simplified linear system obtained by omitting all nonlinear terms from \eqref{eqn:modelled}, which we refer to as the \emph{principal system}. We establish the well-posedness of this principal system using a parabolic lifting approach (see \eqref{eqn:plift}). Although well-posedness can also be derived via an elliptic lifting method (cf. \eqref{eqn:lift}), the latter requires stronger regularity conditions on the boundary data. To generalize our results, we first develop the well-posedness theory using parabolic lifting and then return to the elliptic lifting framework when discussing stability, as the results remain valid in that context. Following this, we apply the boundary feedback control strategy proposed in \cite{BV13, IM15} to study the exponential stability of the principal system. Finally, leveraging the regularity of the solution to the principal system with nonhomogeneous data and applying the Banach fixed-point theorem, we establish the null stabilization of the full nonlinear system \eqref{eqn:nlmodel}.

\subsection{Organization of paper} The remainder of the article is structured as follows: Section \ref{sec:N-P} presents the necessary preliminary notations along with some well-known inequalities that will be used throughout the paper. In Section \ref{sec: Well-R}, we study the well-posedness and regularity of the principal system with boundary and nonhomogeneous data. Section \ref{sec:SLS} is dedicated to establishing the stabilization results for the principal (linearized) system. Finally, in Section \ref{sec:SNlS}, we address the stabilization of the full nonlinear system using the fixed point theorem.

\section{Notations and Preliminaries} \label{sec:N-P}
We denote $L^2(\Omega) := L^2$, the space of equivalence classes of Lebesgue measurable, square-integrable functions on $\Omega$, equipped with the standard norm $\|\cdot\|$ and inner product $(\cdot, \cdot)$.
Likewise, $L^2(\Gamma_1)$ denotes the space of equivalence classes of Lebesgue measurable, square-integrable functions defined on $\Gamma_1$, equipped with the inner product
$$
(y(x), z(x))_{\Gamma_1} = \int_{\Gamma_1} y(x) z(x) \, d\sigma,
$$
where $\sigma$ denotes the Lebesgue surface measure on the boundary $\Gamma_1$.
 By $H^m(\Omega)=:H^m,$ $m\in\mathbb{N}$, we 
denote the standard Sobolev spaces on $\Omega$ endowed with standard norms, denoted by $\|\cdot\|_{H^m}.$ 
 For any Banach space $X$ with the norm $\|\cdot\|_X$, the Bochner space $L^p(0,\infty; X)$ is defined as
$$L^p(0,\infty; X):=\left\{ g:[0, \infty)\to X\Big| \|g\|_{L^p(0,\infty; X)}:=\Big(\int_0^\infty\|g(t)\|^p_X dt\Big)^\frac{1}{p}<\infty\right\} \text{ for } 1\leq p<\infty.$$
We denote the trace space $$H^{s,r}((0,\infty)\times\Gamma)=L^2(0,\infty; H^s(\Gamma))\cap H^r(0,\infty; L^2(\Gamma)), \text{ $s, r$ are positive real numbers}.$$
Next, we recall the result (cf. \cite[Theorem 10.1.5]{BS87}, \cite{CNT24}) that there exists an orthonormal basis $\lbrace\vp_m\rbrace_{m\in  \mathbb{N}}$ of $L^2$ and sequence of positive real numbers $\lbrace \lambda_m\rbrace_{m \in \mathbb{N}}$ with $\lambda_m\rightarrow \infty$ as $m\rightarrow \infty,$ such that 
\begin{equation}\label{eqevlaplace}
\begin{aligned}
& 0<\lambda_1\leq \lambda_2\leq \ldots \leq \lambda_m\leq\ldots, \\
&A\vp_m:=-\Delta \vp_m =\lambda_m \vp_m,  \text{ in }\Omega, \\
&\vp_m \;\in \mathbf{D}(A)=\left\{y\in H^2(\Omega): y=0, \text{ on }\Gamma_1, \, \frac{\partial y}{\partial n}= 0, \text{ on }\Gamma_2\right\},
\end{aligned}
\end{equation}
where $n$ is the unit drawn normal to the boundary $\Gamma_2$. Each eigenvalue $\lambda_m$ is counted according to its multiplicity. To proceed with our analysis in greater depth, we now introduce a set of assumptions that will underpin the subsequent results.

\begin{Remark}\label{rem:eig-N}
We observe that for any given $\omega > 0$, there exist a natural number $N_\omega \in \mathbb{N}$ and a constant $\epsilon_\omega > 0$ such that
$$
-\eta\lambda_j + \omega + \epsilon_\omega < 0, \  \text{ for all }\  j = N_\omega + 1, N_\omega + 2, \ldots,
$$
where $\eta>0$ is the diffusion coefficient appearing in \eqref{eqn:modelled}. 
\end{Remark}

We now impose the following assumption on the normal derivatives of the eigenfunctions of the operator $A$.
\begin{itemize}
\item[($\mathbf{A1}$)] For a given $\omega>0$, the system $\left\lbrace\frac{\partial\vp_i}{\partial n}, \, i=1, 2, \ldots, N_\omega\right\rbrace$ is linearly independent on $L^2(\Gamma_1).$
\end{itemize}
From now onward, we take $\omega>0$ and it will be fixed in Section \ref{sec:SLS}.
 Inspired by the idea of \cite{BV13}, we introduce the following feedback control 
\begin{align}\label{fcontrol}
u=\lambda_1\sum_{j=1}^{N_\omega}\mu_j( w, \varphi_j)\Phi_j,\  \text{ on }\ \Sigma_1,
\end{align}
where
\begin{align} \label{eqdef:Phi_j}
\Phi_i=\sum_{j=1}^{N_\omega} a_{ij}\frac{\partial\varphi_i}{\partial n}, \  i=1, 2,\ldots ,N_\omega,
\end{align}
with $\{a_{ij}\}_{i,j=1}^{N_\omega}$ chosen such that $$\sum_{k=1}^{N_\omega} a_{ij} \left(\frac{\partial\varphi_j}{\partial n}, \frac{\partial\varphi_k}{\partial n}\right)_{\Gamma_1}=\delta_{ij} = \begin{cases} 1 & \text{ if } i=j, \\ 0 & \text{ otherwise} \end{cases}, \, \, i, j = 1, 2, \ldots, N_\omega.$$ Moreover,  \begin{align}
\mu_j=\frac{k+\eta\lambda_j}{k+\eta(\lambda_j-\lambda_1)}. 
\end{align}
where the constant $k$ will be made precise later. We note that by assumption ($\mathbf{A1}$), the Gram matrix $\Big(\left(\frac{\partial\varphi_j}{\partial n}, \frac{\partial\varphi_k}{\partial n}\right)_{\Gamma_1}\Big)_{j,k=1}^{N_\omega}$ is invertible. Therefore, the coefficients $a_{ij}$ and $\pp_i$ are well-defined. As a result, the feedback controller $u$ given in \eqref{fcontrol} is also well-defined. We now recall some standard inequalities that will be used in subsequent sections.

\begin{Lemma}[Young's inequality.] For any non-negative real number $a,b$ and for any $\varepsilon>0,$ the following inequalities hold:
\begin{align} \label{eqPR-YoungIneq}
	ab \le \frac{\varepsilon a^2}{2} + \frac{b^2}{2\varepsilon} \text{ and } ab \le \frac{a^p}{p}+\frac{b^q}{q},
\end{align}
for any $p,q>1$ such that $\frac{1}{p}+\frac{1}{q}=1.$
\end{Lemma}
\begin{Lemma}[Generalized H\"older's inequality] \label{ppsPR-GenHoldIneq}
	Let $f\in L^p(\Omega),$ $g\in L^q(\Omega),$ and $h\in L^r(\Omega),$ where $1\le p, q ,r\leq \infty$ are such that $\frac{1}{p}+\frac{1}{q}+\frac{1}{r}=1.$ Then $fgh\in L^1(\Omega)$ and 
	\begin{align*}
		\|fgh\|_{L^1(\Omega)} \le \|f\|_{L^p(\Omega)} \|g\|_{L^q(\Omega)} \|h\|_{L^r(\Omega)}.
	\end{align*} 
\end{Lemma}
\begin{Lemma}[Agmons' inequality {\n\cite{Agmon10}}] \label{lemVB:AgmonIE}
	Let $\Omega$ be a bounded domain in $\mathbb{R}^d,$  and let $f \in H^{s_2}(\Omega).$ Let $s_1,s_2$ be such that $s_1<\frac{d}{2}<s_2.$ If for $0<\theta<1,$ $\frac{d}{2}=\theta s_1+(1-\theta)s_2,$ then there exists a positive constant $C_a=C_a(\Omega)$ such that
	\begin{equation} \label{eqVB-AgmonIE}
		\|f\|_{L^\infty(\Omega)} \le C_a \|f\|_{H^{s_1}(\Omega)}^{\theta} \|f\|_{H^{s_2}(\Omega)}^{1-\theta}.
	\end{equation}
\end{Lemma}
In the following lemma, we restate the Sobolev embedding theorem (see \cite[Theorem 2.4.4]{Kes}) as it applies in our context.
\begin{Lemma}[Sobolev embedding] \label{lemPR:SobEmb}
	Let $\Omega$ be an open bounded domain in $\mathbb{R}^d$ of class $C^1$  with $d\in \mathbb{N}.$ Then, we have the following continuous inclusion:
	\begin{itemize}
		\item[$(a)$] $H^1(\Omega)\hookrightarrow L^q(\Omega)$ for $d>2,$ where $1\leq q\leq \frac{2d}{d-2},$
		\item[$(b)$] $H^1(\Omega) \hookrightarrow L^q(\Omega)$ for all $d=2, 1\leq q<\infty.$ 
	\end{itemize}
\end{Lemma}
In the next theorem, we recall the Gagliardo-Nirenberg inequality in the case of a bounded domain with a smooth boundary.
\begin{Lemma}[Gagliardo-Nirenberg inequality {\normalfont\cite{Nirenberg1959}}] \label{thm:G-NinBdd}
	Let \( \Omega \subset \mathbb{R}^d \) be a bounded domain with a smooth boundary. For any \( u \in W^{m,q}(\Omega) \), and for any integers \( j \) and \( m \) satisfying \( 0 \leq j < m \), the following inequality holds:
	\[
	\| D^j u \|_{L^p(\Omega)} \leq C \|  u \|_{W^{m,q}(\Omega)}^\theta \| u \|_{L^r(\Omega)}^{1-\theta} ,
	\]
	where \( D^j u \) denotes the \( j \)-th order weak derivative of \( u \), and:
	\[
	\frac{1}{p} = \frac{j}{d}+\theta \left( \frac{1}{q} - \frac{m}{d} \right) +  \frac{(1-\theta)}{r},
	\]
	for some constant \( C \) that depends on the domain \( \Omega \) but not on \( u \), where \( \theta \in [0,1] \).
\end{Lemma}

\begin{Lemma}[General Gronwall inequality \cite{Canon99}] \label{lem:Gronwall} 
	Let $f,g, h$ and $y$ be four locally integrable non-negative functions on $[t_0,\infty)$ such that
	\begin{align*}
		y(t)+\int_{t_0}^t f(s)ds \le C+ \int_{t_0}^t h(s)ds + \int_{t_0}^t g(s)y(s)ds\  \text{ for all }t\ge t_0,
	\end{align*}
	where $C\ge 0$ is any constant. Then 
	\begin{align*}
		y(t)+\int_{t_0}^t f(s)ds \le \left(C+\int_{t_0}^t h(s)ds\right) \exp\left( \int_{t_0}^t g(s)ds\right) \text{ for all }t\ge t_0.
	\end{align*}
\end{Lemma}
We now present a version of the nonlinear generalization of Gronwall’s inequality, which will be instrumental in our subsequent analysis.
\begin{Theorem}[Nonlinear generalization of Gronwall’s inequality {\cite[Theorem 21]{DragSS}}]\label{lem:NonlinGronwall}
Let $\zeta$ be a non-negative function that satisfies the integral inequality
\begin{align*}
	\zeta(t)\le C+ \int_{t_0}^t \left(  a(s)\zeta(s) + b(s)\zeta^{\vartheta}(s) \right) ds, \ C\ge 0, \ 0\le \vartheta<1,
\end{align*}
where $a$ and $b$ are locally integrable non-negative functions on $[t_0,\infty).$ Then, the following inequality holds: 
\begin{align*}
	\zeta(t) \le \left\lbrace C^{1-\vartheta} \exp{\left[(1-\vartheta)\int_{t_0}^ta(s)ds\right]} + (1-\vartheta)\int_{t_0}^t b(s) \exp{\left[(1-\vartheta)\int_{s}^t a(r)dr \right]} ds\right\rbrace^{\frac{1}{1-\vartheta}}.
\end{align*}
\end{Theorem}
\section{Well-posedness and Regularity} \label{sec: Well-R}
In this section, we begin by introducing the principal system associated with \eqref{eqn:nlmodel} and establish its well-posedness and regularity properties. These results are derived by considering $u$, not as a feedback control, but as prescribed boundary data belonging to an appropriate trace space (i.e., a time-dependent Sobolev space). Later, for the purpose of stabilization, we employ an elliptic lifting. However, the results obtained in this section remain valid under such a lifting framework as well.

We now consider the principal system given by:
\begin{equation}\label{eqn:lmodel}
\left\{
\begin{aligned}
& \x_t-\eta\Delta\x+\beta\gamma\x-\int_0^t e^{-\delta(t-s)}\Delta\x(s) \, ds = f,\ \text{ in } \ Q ,\\
&   \x(x,t)=\ub(x,t) \, \text{ for all } \,(x,t)\in \Sigma_1 , \quad  \frac{\partial \x}{\partial n}= 0, \  \text{ on }\ \Sigma_2,\\&\x(x,0)=w_0(x):=\y_0(x)-\y_\infty(x) \ \text{ for all } \ x\in \Omega.
\end{aligned}
\right.
\end{equation}Let us introduce the Dirichlet map $D_p$ as follows:$$H^{\frac{1}{2},\frac{1}{4}}(\Sigma_1) \ni D_p(u)=p\in  L^\infty(0,\infty; L^2(\Omega))\cap L^2(0,\infty; H^1(\Omega))\cap H^1(0,\infty; H^{-1}(\Omega)), $$
where $p$ satisfies the following:
\begin{equation}\label{eqn:plift}
\left\{
\begin{aligned}
&\partial_tp-\eta\Delta p=0, \ \text{ in }\ Q,\\
&p = u,\  \text{ on }\ \Sigma_1, \  \frac{\partial p}{\partial n} =0,\  \text{ on }\ \Sigma_2,\\
& p(x,0)=0, \ \text{ for all } \ x \in \Omega.
\end{aligned}
\right.
\end{equation}
Note that $D_p$ is well-defined due to \cite[Section 15.5]{LioMag-V2}, also the following estimate holds: for all $t\geq 0$,
\begin{align}\label{weak_p}
\sup_{t\in[0,\infty)}\|p(t)\|^2+\eta\int_0^t\|p(s)\|^2_{H^1} ds\leq C\|u\|^2_{H^{\frac{1}{2},\frac{1}{4}}((0,\infty)\times\Gamma_1)},
\end{align}
for some constant $C=C(\eta, Tr)>0.$

\begin{Theorem}\label{existence_weak}
For given $w_0\in L^2(\Omega)$, $f\in L^2(0,\infty;L^2(\Omega)),$ and $u\in H^{\frac{1}{2},\frac{1}{4}}(\Sigma_1)$, the problem \eqref{eqn:lmodel} has a unique solution $w$ such that 
$$w\in L^\infty(0,\infty;L^2(\Omega))\cap L^2(0,\infty;H^1(\Omega))\cap H^1(0,\infty; H^1(\Omega)').$$
\end{Theorem}

\begin{proof}
To establish the existence of a weak solution, we first aim to impose homogeneous (zero) boundary conditions in \eqref{eqn:lmodel}. To this end, we define a new variable $z = \x - p$, where $p$ is the lifting function satisfying \eqref{eqn:plift}. Substituting this into \eqref{eqn:lmodel}, we obtain:
\begin{equation}\label{equ: elift_lin_model}
\left\{ 
\begin{aligned}
&z_t = \eta\Delta\zb-\beta\gamma\zb+\left(\frac{1}{\eta}-\beta\gamma\right) p+\int_0^t e^{-\delta(t-s)}\Delta\zb(\cdot,s) \, ds \\&\qquad + \frac{\delta}{\eta}\int_0^t e^{-\delta(t-s)}p(\cdot,s) \, ds +f,\  \text{ in } \ Q,\\
&\zb=0,\  \text{ on } \Sigma_1, \ \frac{\partial\zb}{\partial n}=0,\  \text{ on }\ \Sigma_2,\\
&\zb(x,0)=w_0 \ \text{ for all } \ x\in \Omega.
\end{aligned}\right.
\end{equation} 
Using \eqref{weak_p}, one can easily see that $(\frac{1}{\eta}-\beta\gamma) p+\frac{\delta}{\eta}\int_0^t e^{-\delta(t-s)}p(\cdot,s) \, ds\in L^2(0,\infty; H^1(\Omega))$. Then, using \cite[Theorem 2.1]{CM06}, we conclude that exist a unique solution $z$ of \eqref{equ: elift_lin_model} such that $z\in L^\infty(0,\infty; L^2(\Omega))\cap L^2(0,\infty; H^1(\Omega))\cap H^1(0,\infty; H^{-1}(\Omega))$, which yields the existence of solution $w$ of \eqref{eqn:lmodel}.
\end{proof}
We now aim to establish a regularity result for the system \eqref{eqn:lmodel} by leveraging the regularity result for the lifting equation \eqref{eqn:plift}. To this end, we first recall \cite[Lemma 4.7]{AGVM25}:
\begin{Lemma}
For a given $u\in H^{\frac{3}{2}, \frac{3}{4}}(\Sigma_1),$ the system \eqref{eqn:plift} admits a unique solution $p$ such that $$p\in L^\infty(0,\infty; H^1(\Omega))\cap L^2(0,\infty; H^2(\Omega))\cap H^1(0,\infty; L^2(\Omega)).$$ Moreover, $p$ satisfies the following estimate for all $t\geq 0$:
\begin{align}
\sup_{t\in[0,\infty)}\|p(t)\|^2_{H^1(\Omega)}+\eta\int_0^t\|p(s)\|^2_{H^2(\Omega)} ds\leq C\|u\|_{H^{\frac{3}{2}, \frac{3}{4}}(\Sigma_1)},
\end{align}
for some constant $C=C(\eta, Tr)>0.$
\end{Lemma}

\begin{Theorem}\label{reg_sol}
For given $w_0\in H^1(\Omega)$, $u\in$ $H^{\frac{3}{2}, \frac{3}{4}}(\Sigma_1)$,  and $f\in L^2(0,\infty;L^2(\Omega)),$ the system \eqref{eqn:lmodel} admits a unique solution $$w\in L^\infty(0,\infty; H^1(\Omega))\cap L^2(0,\infty; H^2(\Omega))\cap L^{2(\kappa+1)}(0, \infty; L^{6(\kappa+1)}(\Omega))\cap H^1(0,\infty; L^2(\Omega)),$$ satisfying
   \begin{align}\label{reg_est}
      & \|w\|_{L^\infty(0,\infty; H^1)}+\|w\|_{L^2(0,\infty;H^2)}+\|w\|_{H^1(0,\infty; L^2)}+\|w\|_{L^{2(\kappa+1)}(0, \infty; L^{6(\kappa+1)})} \no\\ &\qquad
      \leq C\left(\|w_0\|_{H^1}+\|u\|_{H^{\frac{1}{2}, \frac{1}{4}}(\Sigma_1)}+\|f\|_{L^2(0,\infty;L^2)}\right).
   \end{align}
\end{Theorem}
\begin{proof}
 We prove the theorem using the Faedo-Galerkin approximation method.
 \vskip 0.2 cm
\noindent\emph{Faedo-Galerkin approximation:}  Let $\{\vp_1,   \ldots, \vp_n, \ldots\}$ be the set of  orthonormal basis in $L^2(\Omega)$ described in Section \ref{sec:N-P}. For any $n \in \mathbb{N},$ define the finite dimensional space $H_n$ as  $H_n:=\mathrm{span}\{\vp_1,\ldots, \vp_n\},$ and projection operators $P_n: L^2(\Omega) \to H_n$ as $P_n w=\displaystyle\sum_{i=1}^n (w, \vp_i)\vp_i.$ Note that $H_n \subset H^1(\Omega)$ and any element $v_n$ in $H_n$ satisfies $v_n=0$ on $\Gamma_1$ in the sense of trace.  Now, we are interested in finding a function $\zb_n:[0,T]\rightarrow H_n$ such that $\displaystyle \zb_n(x,t)=\sum_{i=1}^{n}a_n^i(t)\vp_i(x),$ where $n$ being a fixed positive integer. The smooth functions $a_n^i(t)$, $0\leq t\leq T,$ and $i=1,2,\ldots, n$, are chosen such that $a_n^i(0)=(z_0,\vp_i)$ and 
\begin{align}\label{equ:gal1}
\langle z'_n(t), v_n\rangle+&\eta(\nabla z_n(t),\nabla v_n) +((K\ast\nabla z_n)(t),\nabla v_n)+\beta\gamma(z_n(t),v_n)\no\\&=\frac{\delta}{\eta}((K\ast p_n)(t), v_n)+\left(\frac{1}{\eta}-\beta\gamma\right)(p_n(t), v_n)+(f_n(t), v_n),
\end{align}
for all $v_n\in H_n$, and for a.e. $0\leq t\leq T$. In the above expression $f_n=P_n f$ and $p_n=P_np$, where $P_n$ being the projection operator as defined above.  Also, for the notational convenience, we have denoted the kernel 
$K(t)= e^{-\delta t}$ and $(K\ast\zb)(t)=\int_0^t e^{-\delta(t-s)}\zb(s) ds.$
\vskip 0.2 cm
\noindent
\emph{A priori estimates:} Taking $v_n=Az_n$ in \eqref{equ:gal1}, we have
\begin{align}\label{ei1}
\frac{1}{2}\frac{d}{dt}\|\nabla z_n(t)\|^2 & +\eta\|\Delta z_n(t)\|^2+\beta\gamma\|\nabla z_n(t)\|^2+((K\ast\Delta z_n)(t),\Delta z_n(t))\no\\& =\frac{\delta}{\eta}((K\ast p_n)(t), Az_n(t))+\left(\frac{1}{\eta}-\beta\gamma\right)(p_n(t), Az_n(t))+(f_n, Az_n(t)).
\end{align}
Now, using the estimate  \eqref{weak_p}, H\"older's and Young's inequalities, we estimate the terms in right hand side of \eqref{ei1} as
\begin{align*}
\bigg|\frac{\delta}{\eta}((K\ast p_n), Az_n)\bigg|&=\frac{\delta}{\eta}|((K\ast \nabla p_n), \nabla z_n)\big|\leq C\|\nabla p\|^2+\frac{\beta\gamma}{4}\|\nabla z_n\|^2,\\
\left(\frac{1}{\eta}-\beta\gamma\right)(p_n, Az_n)&\leq C\|\nabla p\|^2+\frac{\beta\gamma}{4}\|\nabla z_n\|^2, \\
(f_n, Az_n)&\leq C\|f\|^2+\frac{\eta}{2}\|\Delta z_n\|^2.
\end{align*}
Then, substituting all these estimates in \eqref{ei1}, we get
\begin{align}
&\frac{d}{dt}\|\nabla z_n(t)\|^2+\eta\|\Delta z_n(t)\|^2+\beta\gamma\|\nabla z_n(t)\|^2\leq C(\|f(t)\|^2+\|\nabla p(t)\|^2),
\end{align}
for some constant $C$ independent of $T$. Next, integrating the inequality from $0$ to $t$, and then applying Lemma \ref{lem:Gronwall} along with the estimate \eqref{weak_p}, we obtain
\begin{align}\label{reg_unif_est}
&\|\nabla z_n(t)\|^2+\eta\int_0^t \left(\|\Delta z_n(t)\|^2+\beta\gamma\|\nabla z_n(t)\|^2\right) dt\nonumber\\& \leq \|w_0\|^2_{H^1}+C\left(\|f\|^2_{L^2(0,\infty; L^2)}+\|u\|^2_{H^{\frac{3}{2},\frac{3}{4}}(\Sigma_1)}\right),
\end{align}
for all $t>0.$
Thus, we have 
\begin{align}\label{eqn-ener-est}
&	\sup_{t \in [0,\infty)}\|\nabla z_n(t)\|^2+\eta\int_0^\infty \left(\|\Delta z_n(t)\|^2+\beta\gamma\|\nabla z_n(t)\|^2\right) dt\nonumber\\& \leq \|w_0\|^2_{H^1}+C\left(\|f\|^2_{L^2(0,\infty; L^2)}+\|u\|^2_{H^{\frac{3}{2},\frac{3}{4}}(\Sigma_1)}\right).
\end{align}
  By using the elliptic regularity theorem (see \cite[Theorem 3.1.2.1]{PG85}), we have 
   \begin{align*}
       z_n\in L^\infty(0, \infty; H_n)\cap L^2(0,\infty; H^2(\Omega)).
   \end{align*}
   \vskip 0.2 cm
   \noindent
\emph{Time derivative estimate:}   Now, to estimate the time derivative, we choose $v_n = z_n'$ in equation \eqref{equ:gal1}, yielding
   \begin{align*}
       &\|z_n'(t)\|^2+\frac{\eta}{2}\frac{d}{dt}\|\nabla z_n(t)\|^2+\frac{\beta\gamma}{2}\frac{d}{dt}\|z_n(t)\|^2 \no\\&  =((K\ast\Delta z_n)(t), z_n'(t))+\frac{\delta}{\eta}((K\ast p_n)(t), z_n'(t))+\left(\frac{1}{\eta}-\beta\gamma \right)(p_n(t), z_n'(t))+(f_n(t), z_n'(t)),
   \end{align*}
for a.e. $t\in[0,T]$.   Choosing $\Tilde{\eta}=\min\{\eta, \beta\gamma\}$,
   and using H{\"o}lder's and Youngs' inequalities, we obtain
   \begin{align*}
       \Tilde{\eta}\frac{d}{dt}\|z_n(t)\|^2_{H^1}+\|z_n'(t)\|^2&\leq C\left(\|(K\ast\Delta z_n)(t)\|^2+\|(K\ast p_n)(t)\|^2+\|p_n(t)\|^2+\|f_n(t)\|^2\right).
      \end{align*}
Integrating the above equation from $0$ to $t,$ and using \eqref{weak_p} and boundedness of the kernel $K(t)=e^{-\delta t}$, we find
\begin{align}\label{reg_time_est}
\Tilde{\eta}\|z_n(t)\|_{H^1}^2+\int_0^t\|z_n'(t)\|^2dt\leq C\bigg(\int_0^t\|\Delta z_n(s)\|^2\, ds+\|u\|_{H^{\frac{1}{2},\frac{1}{4}}(\Sigma_1)}+\|f\|^2_{L^2(0,\infty; L^2)}\bigg).
  \end{align}
 Then, using the estimate \eqref{reg_unif_est}, we obtain the right-hand side of the inequality \eqref{reg_time_est} is independent of $n$, and the left-hand side is bounded for any $t\geq 0$; consequently, it follows that $$z_n'\in L^2(0,\infty; L^2(\Omega)).$$
  \vskip 0.2 cm
 \noindent
\noindent \emph{$L^p$ estimate:} Now, to derive the $L^{2(\kappa+1)}(0, \infty; L^{6(\kappa+1)}(\Omega))$ bound on $z_n$, we select the test function $v_n = P_n(|z_n|^{2\kappa}z_n)$ in equation \eqref{equ:gal1}. Utilizing the self-adjointness and commutativity properties of the projection operator $P_n$ with the operator $A$, we obtain
    \begin{align*}
&(z'_n(t),|z_n(t)|^{2\kappa}z_n(t)) +\eta(\nabla z_n,\nabla(|z_n(t)|^{2\kappa}z_n(t)))
+\beta\gamma(z_n(t),|z_n(t)|^{2\kappa}z_n(t))\no\\& \quad+((K\ast\nabla z_n)(t),\nabla(|z_n(t)|^{2\kappa}z_n(t)))\no\\& \hspace{1cm} =\frac{\delta}{\eta}((K\ast p_n)(t), |z_n(t)|^{2\kappa}z_n(t)) +\left(\frac{1}{\eta}-\beta\gamma\right)(p_n(t), |z_n(t)|^{2\kappa}z_n(t))+(f_n, |z_n(t)|^{2\kappa}z_n(t)),
  \end{align*} 
 for a.e. $t\in[0,T]$. Then, integration by parts gives
  \begin{equation}\label{IP1}
  \begin{aligned}
      &\frac{1}{2(\kappa+1)}\frac{d}{dt}\|z_n(t)\|^{2(\kappa+1)}_{L^{2(\kappa+1)}}+\eta(2\kappa+1)\||z_n(t)|^\kappa\nabla z_n(t)\|^2 +\beta\gamma\|z_n(t)\|^{2(\kappa+1)}_{L^{2(\kappa+1)}}\\
       & \qquad + (2\kappa+1)\Big( (K *\nabla z_n)(t) , |z_n(t)|^{2\kappa} \nabla z_n(t)\Big)  \\
      & \, \leq \frac{\delta}{\eta}((K\ast p_n)(t), |z_n(t)|^{2\kappa}z_n(t))+\left(\frac{1}{\eta}-\beta\gamma\right)(p_n(t), |z_n(t)|^{2\kappa}z_n(t))+(f_n(t), |z_n(t)|^{2\kappa}z_n(t)).
  \end{aligned}
  \end{equation}
 We can estimate the terms in the right hand side of \eqref{IP1} as follows:
 \begin{align*}
     \Big|\frac{\delta}{\eta}\big((K\ast p_n), |z_n|^{2\kappa}z_n\big)\Big|&=\Big|\frac{\delta}{\eta}\big( |z_n|^\kappa (K\ast p_n), |z_n|^{\kappa}z_n\big)\Big|\\&\leq \frac{\delta}{\eta}\||z_n|^{\kappa}z_n\|_{H^1_0}\||z_n|^\kappa (K\ast p_n)\|_{H^{-1}}\\&\leq \frac{\delta}{\eta}\|\nabla(|z_n|^{\kappa}z_n)\|\||z_n|^\kappa (K\ast p_n)\|_{L^{\frac{2\kappa+2}{2\kappa+1}}}\\&\leq C(\kappa+1)\||z_n|^{\kappa}\nabla z_n\|\|z_n|^\kappa (K\ast p_n)\|_{L^{\frac{2\kappa+2}{2\kappa+1}}}\\&\leq 
C(\kappa+1)\||z_n|^{\kappa}\nabla z_n\|\||z_n|^\kappa\|_{L^\frac{2(\kappa+1)}{\kappa}}\|(K\ast p_n)\|\\&\leq\frac{(2\kappa+1)\eta}{6}\||z_n|^{\kappa}\nabla z_n\|^2+C\|z\|^{2\kappa}_{L^{2(\kappa+1)}}\|(K\ast p_n)\|^2,
  \end{align*}
where we have used the fact that $L^{\frac{2\kappa+2}{2\kappa+1}}(\Omega)\hookrightarrow H^{-1}(\Omega)$ for the values of $k$ given in \eqref{eqdef:delta}.    Similarly, we estimate the last two terms in the right hand side of \eqref{IP1} and get
 \begin{align*}
 \left|\left(\frac{1}{\eta}-\beta\gamma\right)(p_n, |z_n|^{2\kappa}\sz_n)\right|&\leq\frac{(2\kappa+1)\eta}{6}\||z_n|^{\kappa}\nabla z_n\|^2+C\|z_n\|^{2\kappa}_{L^{2(\kappa+1)}}\|p_n\|^2, \\
    |(f_n, |z_n|^{2\kappa}z_n)|&\leq\frac{(2\kappa+1)\eta}{6}\||z_n|^{\kappa}\nabla z_n\|^2+C\|z_n\|^{2\kappa}_{L^{2(\kappa+1)}}\|f_n\|^2.
  \end{align*}
   Utilizing these estimates in \eqref{IP1}, we obtain
   \begin{equation}\label{IP2}
  \begin{aligned}
    &\frac{1}{2(\kappa+1)}\frac{d}{dt}\|z_n(t)\|^{2(\kappa+1)}_{L^{2(\kappa+1)}}+\frac{\eta(2\kappa+1)}{2}\||z_n(t)|^\kappa\nabla\sz_n(t)\|^2+\beta\gamma \|z_n(t)\|^{2(\kappa+1)}_{L^{2(\kappa+1)}}    \\
    & \qquad + (2\kappa+1)\Big( (K *\nabla z_n)(t) , |z_n(t)|^{2\kappa} \nabla z_n(t)\Big)  \\
    & \quad\leq C\big(\|p_n(t)\|^2+\|f_n(t)\|^2\big) \|\sz_n(t)\|^{2\kappa}_{L^{2(\kappa+1)}}.
  \end{aligned}
  \end{equation}
  Define $\mathcal{Z}(t)=\|z_n(t)\|^{2(\kappa+1)}_{L^{2(\kappa+1)}}$.  Note that, an application of \cite[Lemma A.2]{MTMSSS} leads to
  \begin{align*}
  &	\int_0^t \Big( (K *\nabla z_n)(\tau) , |z_n(\tau)|^{2\kappa} \nabla z_n(\tau)\Big) d\tau \nonumber\\&= \int_0^t \Big( |z_n(\tau)|^{2\kappa}(K *\nabla z_n)(\tau) ,  \nabla z_n(\tau)\Big) d\tau\nonumber\\&=
  	\int_0^t\int_{\Omega}|z_n(\tau,x)|^{2\kappa}\int_0^{\tau}K(\tau-s)\nabla z_n(s,x)ds\cdot\nabla z_n(\tau,x)dxd\tau\nonumber\\&=\int_{\Omega}	\int_0^t|z_n(\tau,x)|^{2\kappa}\int_0^{\tau}K(\tau-s)\nabla z_n(s,x)\cdot\nabla z_n(\tau,x)dsd\tau dx\ge 0.
  \end{align*} 
The change in the order of integration is justified by the equivalence of all norms in finite-dimensional spaces; see equation \eqref{eqn-ener-est} for an example of such a finite energy estimate.    Then, integrating \eqref{IP2} from $0$ to $t,$ and using above displayed inequality, one obtains the following inequality:
  \begin{align*}
      \mathcal{Z}(t)\leq\mathcal{Z}(0)+C\int_0^t\big(\|p_n(t)\|^2+\|f_n(s)\|^2\big)\mathcal{Z}(s)^\frac{2\kappa}{2(\kappa+1)} \, ds.
  \end{align*}
  Since $0<\frac{2\kappa}{2(\kappa+1)}<1,$ using a nonlinear generalization of Gronwall's inequality [see Theorem \ref{lem:NonlinGronwall}], we have 
  \begin{align*}
      \mathcal{Z}(t)\leq \Big\{\mathcal{Z}(0)^\frac{1}{(\kappa+1)}+\frac{1}{(\kappa+1)}\int_0^t\big(2\|p_n(t)\|^2+\|f_n(s)\|^2\big) \, ds\Big\}^{(\kappa+1)},
  \end{align*}
  that is, 
  \begin{align}\label{IP3}
\|z_n(t)\|^2_{L^{2(\kappa+1)}}&\leq \|w_0\|^2_{L^{2(\kappa+1)}}+ \frac{1}{(\kappa+1)}\Big(\|u\|^2_{H^{\frac{1}{2},\frac{1}{4}}(\Sigma_1)}+\int_0^t\|f(s)\|^2 \, ds\Big)\no\\&\leq \|w_0\|^2_{L^{2(\kappa+1)}} +C\left(\|u\|^2_{H^{\frac{1}{2},\frac{1}{4}}(\Sigma_1)}+\|f\|^2_{L^2(0, \infty, L^2)}\right).
  \end{align}
  Using \eqref{IP3} in \eqref{IP2}, we deduce that 
  \begin{align}\label{IP4}
&\|z_n(t)\|^{2(\kappa+1)}_{L^{2(\kappa+1)}}+\frac{\eta(2\kappa+1)(\kappa+1)}{2}\int_0^t\||z_n(t)|^\kappa\nabla z_n(t)\|^2 dt \no\\& \qquad\leq C\bigg(\|w_0\|^{2(\kappa+1)}_{L^{2(\kappa+1)}}+\|u\|^{2(\kappa+1)}_{H^{\frac{1}{2},\frac{1}{4}}(\Sigma_1)}+\|f\|^{2(\kappa+1)}_{L^2(0,\infty;L^2)}\bigg),
  \end{align}
  for all $0\leq t<\infty.$ Thus, we have 
  \begin{align*}
  	&\sup_{t\ge 0}\|z_n(t)\|^{2(\kappa+1)}_{L^{2(\kappa+1)}}+\frac{\eta(2\kappa+1)(\kappa+1)}{2}\int_0^\infty\||z_n(t)|^\kappa\nabla z_n(t)\|^2 dt \no\\& \qquad\leq C\bigg(\|w_0\|^{2(\kappa+1)}_{L^{2(\kappa+1)}}+\|u\|^{2(\kappa+1)}_{H^{\frac{1}{2},\frac{1}{4}}(\Sigma_1)}+\|f\|^{2(\kappa+1)}_{L^2(0,\infty;L^2)}\bigg).
  \end{align*}
   Now, using the Sobolev embedding, we estimate the following:
  \begin{align}\label{reg_l_est}
      \|z_n\|^{2(\kappa+1)}_{L^{2(\kappa+1)}(0,\infty;L^{6(\kappa+1)})}&=\int_0^\infty\|z_n(t)\|^{2(\kappa+1)}_{L^{6(\kappa+1)}} dt=\int_0^\infty\|z_n(t)^{\kappa+1}\|^2_{L^6} dt\no\\ &\leq C(\kappa+1)^2\int_0^\infty\|z_n(t)^\kappa\nabla z_n(t)\|^2 \, dt\no\\&\leq C(\kappa+1)^2 \bigg(\|w_0\|^{2(\kappa+1)}_{L^{2(\kappa+1)}}+\|u\|^2_{H^{\frac{1}{2},\frac{1}{4}}(\Sigma_1)}+\|f\|^{2(\kappa+1)}_{L^2(0,\infty;L^2)}\bigg),
  \end{align}
  where in the last inequality, we have used the estimate \eqref{IP4}.  From \eqref{reg_unif_est} and \eqref{reg_time_est}, we can invoke the Banach Alaoglu theorem to extract a subsequence still denoted by $\zb_n$ such that
\begin{equation}\label{weak_lim}
\left\{ 
\begin{aligned}
\zb_n\rightarrow \zb \text{ weak* in } L^\infty(0, \infty; H^1(\Omega)),\\
\zb_n\rightarrow \zb \text{ weak in } L^2(0, \infty; H^2(\Omega)),\\
\zb'_n\rightarrow \zb' \text{ weak in } L^2(0, \infty; L^2(\Omega)),
\end{aligned}
\right.
\end{equation}
as $n\to\infty.$
Utilizing the compact embedding 
$H^2(\Omega) \subset
H^1(\Omega)$ and applying the Aubin-Lions compactness lemma \cite[Theorem 5]{JSimon}, we obtain the following strong convergence result:
\begin{align}\label{st_lim}
    \zb_n\rightarrow\zb \text{ strong in } L^2(0, \infty; H^1(\Omega)).
\end{align} 
The strong convergence further implies there exist a subsequence $(\zb_{n_j})$ such that  $\zb_{n_j}(t)\to\zb(t)$ for a.e. $t\in [0, \infty)$ in $L^2(\Omega)$ and $\zb_{n_j}(x,t)\to\zb(x,t)$ for a.e. $(x, t)\in Q.$ Now we can pass to the limit in \eqref{equ:gal1} and find $$z \in L^\infty(0,\infty; H^1(\Omega))\cap L^2(0,\infty; H^2(\Omega))\cap L^{2(\kappa+1)}(0, \infty; L^{6(\kappa+1)}(\Omega))\cap H^1(0,\infty; L^2(\Omega)),$$ as a solution of \eqref{equ: elift_lin_model},  similarly as  in \cite[Theorem 2.1]{CM06}. Using \eqref{weak_p}, it is easy to see that $$w \in L^\infty(0,\infty; H^1(\Omega))\cap L^2(0,\infty; H^2(\Omega))\cap L^{2(\kappa+1)}(0, \infty; L^{6(\kappa+1)}(\Omega))\cap H^1(0,\infty; L^2(\Omega)),$$ and satisfies \eqref{eqn:lmodel}.

To derive the estimate \eqref{reg_est} for $w$, we first obtain the corresponding estimate for $z$, from which the desired bound for $w = z + p$ follows directly by applying the triangle inequality. To this end, we combine the inequalities \eqref{reg_unif_est}, \eqref{reg_time_est}, and \eqref{reg_l_est} to obtain
\begin{align*}
 & \|z_n\|_{L^\infty(0,\infty; H^1)}+\|z_n\|_{L^2(0,\infty;H^2)}+\|z_n\|_{H^1(0,\infty; L^2)}+\|z_n\|_{L^{2(\kappa+1)}(0, \infty; L^{6(\kappa+1)})} \no\\ &\qquad
      \leq C(\|w_0\|_{H^1\cap L^{2(\kappa+1)}}+\|u\|_{H^{\frac{1}{2}, \frac{1}{4}}(\Sigma_1)}+\|f\|_{L^2(0,\infty;L^2)}).
\end{align*}
Note that the right-hand side of the above inequality is independent of $n$. Taking $\displaystyle\liminf_{n \to \infty}$ on both sides and using the weak lower semicontinuity of norms, we obtain the estimate \eqref{reg_est} for $z$. This concludes the proof.
\end{proof} 

\section{Stabilization of Linearized System} \label{sec:SLS} 
In this section, we investigate the stabilizability of the system \eqref{eqn:lmodel} by employing a boundary control $u$ in the feedback form, as specified in \eqref{fcontrol}. To begin, we observe that when $f \equiv 0$ and $u \equiv 0$, the solution $w$ to \eqref{eqn:lmodel} exhibits exponential decay to zero for every initial condition $w_0 \in L^2$, with a fixed decay rate. Our goal is to design a boundary feedback controller for \eqref{eqn:lmodel} such that the resulting solution $w$ decays exponentially to zero at an arbitrary prescribed rate $\omega \in (0, \omega_0-\epsilon),$ where
\begin{align} \label{eqdef:omega_0}
	\omega_0 : = 2 \eta \lambda_1 +\beta\gamma + \frac{\lambda_1^2 \eta^2}{k + \eta (\lambda_{N_\omega}-\lambda_1)},
\end{align}
and here $\epsilon>0$ could be any arbitrary small number. Furthermore, we require that the same controller also ensures stabilization of the nonlinear system \eqref{eqn:nlmodel}. To this end, we introduce the transformations $\tilde{w}(t) := e^{\omega t} w(t)$ and $\tilde{u}(t) := e^{\omega t} \bar{u}(t)$ in \eqref{eqn:lmodel}. Under this change of variables, the pair $(\tilde{w}, \tilde{u})$ satisfies the following modified system:
\begin{equation}  \label{eqn:shift_lmodel}
\left\{
\begin{aligned}
& \tx_t=\eta\Delta\tx+(\omega-\beta\gamma)\tx+\int_0^t e^{-\delta(t-s)}\Delta\tx(s) \, ds  ,\quad \text{ in } Q ,\\
&   \tx=\tu, \quad  \text{ on } \Sigma_1 , \quad  \frac{\partial \tx}{\partial n}= 0, \quad  \text{ on } \Sigma_2,\\&\tx(x,0)=w_0, \, \text{ for all } x\in \Omega.
\end{aligned}
\right.
\end{equation}
 Note that, if the control $\tu$ stabilizes the system \eqref{eqn:shift_lmodel} exponentially then the control $u$ exponentially stabilizes the system \eqref{eqn:lmodel} with decay rate $\omega$. So it is enough to discuss the stabilization of \eqref{eqn:shift_lmodel} with a control $\tu.$
To this end, we introduce the Dirichlet map $D$ as follows: given $\tu$ in a suitable trace space, we denote by $D\tu:=\psi,$ the solution of

 \begin{equation}\label{eqn:lift}
\left\{
\begin{aligned}
&-\eta\Delta \psi+k\psi=0, \quad \text{ in }\Omega,\\
&\psi = \tu, \text{ on }\Gamma_1, \, \frac{\partial \psi}{\partial n} =0, \text{ on }\Gamma_2,
\end{aligned}
\right.
\end{equation} 
where $k>0$ be any fixed positive constant. Now, for any $\omega$ with $\omega+\epsilon<\omega_0,$ where $\omega_0$ is as in \eqref{eqdef:omega_0}, we note that 
\begin{align*}
	& -(\lambda_j\eta+\beta\gamma-\omega)(k+\lambda_j\eta)-(\omega+k-\beta\gamma)\lambda_1\eta  = (\omega -\eta \lambda_j -\eta \lambda_1 -\beta\gamma) (k+\eta\lambda_j -\eta \lambda_1) -\eta^2\lambda_1^2, 
\end{align*}
and therefore for all $j=1,2,\ldots, N_\omega,$
\begin{equation} \label{(c)}
\begin{aligned} 
 & \frac{-(\lambda_j\eta+\beta\gamma-\omega)(k+\lambda_j\eta)-(\omega+k-\beta\gamma)\lambda_1\eta}{k+(\lambda_j-\lambda_1)\eta} +\epsilon \\
 & \qquad = (\omega +\epsilon -\eta \lambda_j -\eta \lambda_1 -\beta\gamma) - \frac{\eta^2\lambda_1^2}{k +\eta (\lambda_j-\lambda_1)} \\
 & \qquad \le \omega+\epsilon - 2\eta\lambda_1 -\beta \gamma - \frac{\eta^2\lambda_1^2}{k +\eta (\lambda_{N_\omega}-\lambda_1)} = \omega +\epsilon -\omega_0 <0.
\end{aligned}
\end{equation}
In the last inequality , we have used \eqref{eqdef:omega_0}. Also, we have
\begin{align} \label{(d)}
 \frac{-\lambda_j(k+\lambda_j\eta)-k\lambda_1\eta}{k+(\lambda_j-\lambda_1)\eta}<0, \text{ for } j=1,2,\ldots, N_\omega .
\end{align}

The Dirichle map is well defined and $D\in\mathcal{L}(H^\frac{1}{2}(\Gamma_1), H^1(\Omega)).$ Set $\sz=\tx-D\tu$. Note that for later purposes, it is convenient to define a feedback control $\tu$ (see \eqref{fcontrol}) in terms of $\sz$. Similar to  \cite{BV13}, we choose a feedback control as
 \begin{align}\label{fcontrol1}
\tu=\lambda_1\sum_{j=1}^{N_\omega} ( \sz,\varphi_j)\Phi_j.
 \end{align}
Indeed, substitution of $\sz=\tx-D\tu$ in \eqref{fcontrol1} yields
\begin{align}\label{fcontrol2}
\tu=\lambda_1\sum_{j=1}^{N_\omega}(\tx,\varphi_j)\Phi_j-\lambda_1\sum_{j=1}^{N_\omega}(\tu, D^\ast\varphi_j)_{\Gamma_1}\Phi_j,
\end{align}
where $\lambda_1$ is the first eigenvalue, $\varphi_j$ eigenfunctions as discussed in \eqref{eqevlaplace}, $\Phi_j$ are as defined in \eqref{eqdef:Phi_j}, $D^\ast$ is the adjoint of $D$, $D^\ast\in \mathcal{L}(H^1(\Omega);H^\frac{1}{2}(\Gamma_1))$, and $(\cdot, \cdot)_{\Gamma_1}$ denotes the inner product on $\Gamma_1$.  Taking the inner product in \eqref{eqn:lift} with $\varphi_i$ by replacing $\tu=\Phi_j$ in \eqref{eqn:lift}, where $\psi=D\Phi_j$ (recall $\varphi_i, \lambda_i$  from \eqref{eqevlaplace})  and using Green's formula, we obtain 
\begin{align*}
&\eta\int_\Omega\nabla (D\pp_j)\cdot\nabla\vp_i-\eta\int_{\Gamma}(\nabla (D\pp_j)\cdot n)\vp_i+k\int_\Omega D\pp_j\vp_i=0,\no\\
&\implies -\eta\int_\Omega D\pp_j\Delta\vp_i+\eta\int_{\Gamma}\frac{\partial\vp_i}{\partial n}D\pp_j+k\int_\Omega D\pp_j\vp_i=0,\no\\
&\implies  (k+\eta\lambda_i)\int_\Omega D\pp_j\vp_i=-\eta\int_{\Gamma_1}\frac{\partial\vp_i}{\partial n}\pp_j \, \text{ for } i, j=1, 2, \ldots, {N_\omega},
\end{align*}
so that
\begin{align}\label{exp1}
( D^\ast\vp_i, \pp_j)_{\Gamma_1}=\int_\Omega D\pp_j\vp_i=-\frac{\eta}{k+\eta\lambda_i}\int_{\Gamma_1}\frac{\partial\vp_i}{\partial n}\pp_j =-\frac{\eta}{k+\eta\lambda_i}\delta_{ij}.
\end{align}
Thus, \eqref{fcontrol2} yields
\begin{align*}
(\tu, D^\ast\vp_i)_{\Gamma_1} &=\Big( \lambda_1\sum_{j=1}^{N_\omega}(\tx,\varphi_j)\Phi_j-\lambda_1\sum_{j=1}^{N_\omega}(\tu, D^\ast\varphi_j)_{\Gamma_1}\Phi_j, D^\ast\vp_i\Big)_{\Gamma_1}\no\\
& = \lambda_1\sum_{j=1}^{N_\omega}(\tx,\varphi_j)(\Phi_j, D^\ast\vp_i)_{\Gamma_1}-\lambda_1\sum_{j=1}^{N_\omega}(\tu, D^\ast\varphi_j)_{\Gamma_1}\langle\Phi_j, D^\ast\vp_i)_{\Gamma_1}.
\end{align*}
In particular, choosing $j=i,$ one can deduce
\begin{align*}
&\sum_{i=1}^{N_\omega}\left(1-\frac{\lambda_1\eta}{k+\lambda_i\eta}\right)(\tu, D^\ast\vp_i)_{\Gamma_1}= -\sum_{i=1}^{N_\omega}\frac{\lambda_1\eta}{k+\lambda_i\eta}(\tx,\varphi_i),
\end{align*}
that is,
\begin{align*}
&(\tu, D^\ast\vp_i)_{\Gamma_1}=-\frac{\lambda_1\eta}{k+\lambda_i\eta-\lambda_1\eta}(\tx,\varphi_i).
\end{align*}
Substituting the above expression in \eqref{fcontrol2}, we arrive at \eqref{fcontrol}.

Now, using $\sz=\tx-D\tu$ in \eqref{eqn:lmodel} with the control $\tu$ defined in \eqref{fcontrol1}, we get
\begin{equation}\label{equ: lift_lin_model}
\left\{ 
\begin{aligned}
&\sz_t = \eta\Delta\sz-(\omega-\beta\gamma)\sz +(\omega+k-\beta\gamma) D\tu -(D\tu)_t+\int_0^t e^{-\delta(t-s)}\Delta\sz(s) \, ds \\ &\qquad+\frac{k}{\eta}\int_0^t e^{-\delta(t-s)}D\tu(s)\, ds  , \text{ in } Q,\\
&\sz=0, \text{ on } \Sigma_1, \, \frac{\partial\sz}{\partial n}=0, \text{ on }\Sigma_2,\\
&\sz(x,0)=w_0, \text{ in }\Omega.
\end{aligned}\right.
\end{equation} 
We now establish the main result of this section, which asserts that the system \eqref{eqn:shift_lmodel} is well-posed and that the control $\tu$ defined in \eqref{fcontrol1} ensures exponential stabilization of the system, provided that $N_\omega$ and $k$ are appropriately chosen. Equivalently, this amounts to proving that the system \eqref{equ: lift_lin_model} is both well-posed and exponentially stable under the same control $\tu$ from \eqref{fcontrol1}.

\begin{Theorem}
For any $\epsilon > 0$, let $\omega \in (0, \omega_0 - \epsilon)$ (with $\omega_0$ defined in \eqref{eqdef:omega_0}), $N_\omega \in \mathbb{N}$ as given in Remark \ref{rem:eig-N}, $k$ satisfying conditions \eqref{(c)} and \eqref{(d)}, and consider $w_0 \in L^2(\Omega)$.
Then, under the assumption ($\mathbf{A1}$), the feedback controller in \eqref{fcontrol1} exponentially stabilizes \eqref{equ: lift_lin_model} in $L^2(\Omega).$ 
\end{Theorem}
\begin{proof}
Let $X^1=\mathrm{lin\ span} \{\vp_i\}_{i=1}^{N_\omega}$, and $P_{N_\omega}$ be the algebric projection of $L^2(\Omega)$ on $X^1$. We set $\sz^1=P_{N_\omega}\sz,$ $ \sz^2=(I-P_{N_\omega})\sz$. Then, we can write $\sz=\sz^1+\sz^2$. Similarly, we can write the initial data $w_0$ as $w_0=P_{N_\omega}w_0+(I-P_{N_\omega})w_0.$ If we represent $\displaystyle \sz^1=\sum_{i=1}^{N_\omega}z_i(t)\vp_i(x)$ and $P_{N_\omega}w_0=\sum_{i=1}^{N_\omega}z_{0i}\vp_i(x)$,
then we can rewrite \eqref{equ: lift_lin_model} as a finite system of decoupled integro-differential equations of the form
\begin{align}\label{finite part0}
    \sum_{i=1}^{N_\omega} z'_i(t)\vp_i&=-\sum_{i=1}^{N_\omega} (\eta\lambda_i+\beta\gamma-\omega)z_i(t)\vp_i+(\omega+k-\beta\gamma)\lambda_1\sum_{j=1}^{N_\omega}\bigg(\sum_{i=1}^{N_\omega} z_i(t)\vp_i, \vp_j\bigg) D\pp_j\no\\& \quad-\lambda_1\sum_{j=1}^{N_\omega}\bigg(\sum_{i=1}^{N_\omega} z'_i(t)\vp_i, \vp_j\bigg) D\pp_j - \int_0^t e^{-\delta(t-s)}\sum_{i=1}^{N_\omega} \lambda_i z_i(s)\vp_i \, ds\no\\& \quad+\frac{k}{\eta}\int_0^t e^{-\delta(t-s)}\lambda_1\sum_{j=1}^{N_\omega}\bigg(\sum_{i=1}^{N_\omega} z_i(s)\vp_i, \vp_j\bigg) D\pp_j.
\end{align}
Now taking inner product with  $\vp_i$, $1\leq i\leq N_\omega$ in \eqref{finite part0}, we find 
\begin{align*}
 \bigg(1- \frac{\lambda_1\eta}{k+\lambda_i\eta}\bigg) z'_i(t)&=\bigg(\omega-\eta\lambda_i-\beta\gamma-\frac{(\omega+k-\beta\gamma)\lambda_1\eta}{k+\eta\lambda_i}\bigg) z_i(t)\no\\&\quad-\bigg(\lambda_i+\frac{k\lambda_1}{k+\lambda_i\eta}\bigg)\int_0^t e^{-\delta(t-s)} z_i(s) \, ds,
\end{align*}
so that 
\begin{align}\label{finite part1}
  z'_i(t)&= \frac{-(\eta\lambda_i+\beta\gamma-\omega)(k+\eta\lambda_i)-(\omega+k-\beta\gamma)\lambda_1\eta}{k+\eta\lambda_i-\lambda_1\eta} z_i(t) \no\\&\quad+ \frac{-\lambda_i(k+\eta\lambda_i)-k\lambda_1}{k+\eta\lambda_i-\lambda_1\eta}\int_0^t e^{-\delta(t-s)} z_i(s) \, ds.
\end{align}
Therefore, the following systems are obtained:
\begin{equation}\label{ffinite}
    \left\{
    \begin{aligned}
       & z'_i(t)= \frac{-(\eta\lambda_i+\beta\gamma-\omega)(k+\eta\lambda_i)-(\omega+k-\beta\gamma)\lambda_1\eta}{k+\eta\lambda_i-\lambda_1\eta} z_i(t)\\ & \qquad \quad - \frac{\lambda_i(k+\eta\lambda_i)+k\lambda_1}{k+\eta\lambda_i-\lambda_1\eta}\int_0^t e^{-\delta(t-s)} z_i(s) \, ds, \, t>0, \\
       & z_i(0)=z_{0i}, \quad\text{ for } i=1, 2, \cdots, {N_\omega},
    \end{aligned}
    \right.
\end{equation}
and
\begin{equation}\label{infinte}
    \left\{
    \begin{aligned}
& z'_i(t)= -(\eta\lambda_i+\beta\gamma-\omega) z_i(t)-\lambda_i\int_0^t e^{-\delta(t-s)} z_i(s) \, ds +(S(\tu)(t), \vp_i), \quad t>0,\\
& z_i(0)=(I-P_{N_\omega})w_0, \quad\text{ for } i={N_\omega}+1, {N_\omega}+2,\cdots,
   \end{aligned}
    \right.
\end{equation}
where $S(\tu)$ is given  by
\begin{align}\label{def_S}S(\tu)=(\omega+k-\beta\gamma)D\tu-(D\tu)_t+\frac{k}{\eta}\int_0^t e^{-\delta(t-s)} D\ub(s) \, ds.\end{align}
From $\tu$ given in \eqref{fcontrol1}, we can estimate $S(\tu)$ as given in  \cite[equation 3.12]{IM15} and get
\begin{align}\label{bdd_S}
  \|S(\tu)(t)\|^2 \leq C\bigg( \left\Vert \frac{d\sz^1(t)}{dt} \right\Vert^2+\|\sz^1(t)\|^2\bigg),  
\end{align}
for some positive constant $C$.
Let us denote 
$$A_i:= \frac{-(\eta\lambda_i+\beta\gamma-\omega)(k+\eta\lambda_i)-(\omega+k-\beta\gamma)\lambda_1\eta}{k+\eta\lambda_i-\lambda_1\eta},
\, \text{ for }\  i= 1, 2, \cdots, N_\omega$$ and $$B_i:=\frac{-\lambda_i(k+\eta\lambda_i)-k\lambda_1}{k+\eta\lambda_i-\lambda_1\eta}, \  \text{ for }\ i= 1, 2, \cdots, {N_\omega}.$$
Then \eqref{(c)} and  \eqref{(d)} yield
\begin{align}\label{con1}
  A_i+\epsilon < 0, \quad B_i < 0,\ \text{ for }\ i=1,2, \cdots, {N_\omega}.  
\end{align}
$$$$
With the setting
\begin{align*}
\mbox{$X_{N_\omega}= \begin{pmatrix} z_1\\ z_2\\ \vdots \\ z_{N_\omega}\end{pmatrix}$, $\mathcal{A}=\begin{pmatrix} A_1 & 0 & \cdots &0\\ 0 & A_2 &\cdots &0\\ \vdots & \vdots& \ddots&\vdots\\
0 &0&\cdots& A_{N_\omega}\end{pmatrix}$, $\mathcal{B}(t)= e^{-\delta t}\begin{pmatrix}B_1&0&\cdots&0\\0&B_2&\cdots&0\\ \vdots&\vdots&\ddots&\vdots\\0&0&\cdots&B_{N_\omega} \end{pmatrix}$}
\end{align*}
we can write the system \eqref{ffinite} as
\begin{align}
    X_{N_\omega}'(t)=\mathcal{A}X_{N_\omega}(t)+\int_0^t\mathcal{B}(t-s)X_{N_\omega}(s) \, ds.
\end{align}
This is a linear integro-differential equation whose exponential stability has been extensively studied; in particular, we shall refer to \cite{VB72}. More precisely, since our kernel $K(t)=e^{-\delta t}$ is positive, we have (see \cite[Chapter IV, Lemma 4.1]{VB76}) $$\Re(\hat{K}(\xi))>0, \ \text{ for all } \ \xi \in \mathbb{C} \ \text{ with }\ \Re(\xi)>-\epsilon, $$
where $\hat{K}$ denotes the Laplace transformation of $K$ and $\Re$ represents the real part.  More precisely, we can choose $\epsilon=\delta$ since $\hat{K}(\xi)=\frac{1}{\delta+\xi}$. Thus, for all $i=1,2,\ldots, {N_\omega},$ and for all $\xi \in \mathbb{C}$ with $\Re(\xi)>-\epsilon,$ from \eqref{con1}, we obtain $$\Re[\xi-A_i-B_i\hat{K}(\xi)]>-\epsilon-A_i-B_i\hat{K}(\xi)>0,$$ and this further implies $(\xi-\mathcal{A}-\hat{\mathcal{B}}(\xi))$  is invertible for all $\xi \in \mathbb{C}$ with $\Re(\xi)>-\epsilon$. Then, relying on the result in  \cite[Corollary 3.3]{VB72}, we found that $X_{N_\omega}$ is exponentially stable, that is, there exists a $\alpha_1>0$ such that
\begin{align}\label{exp_lin}
    \|X_{N_\omega}(t)\|\leq Ce^{-\alpha_1 t}\|X_{N_\omega}(0)\|, \ \text{ that is, } \  \|X_{N_\omega}(t)\|\leq Ce^{-\alpha_1 t}\|w(0)\|.
\end{align}
 Furthermore, by \cite[Corollary 3.3]{VB72}, we know that $X_{N_\omega}$ is integrable over $[0, \infty)$. Then, invoking \cite[Proposition 3.1]{AR04}, we deduce that the time derivative of $X_{N_\omega}$ also exhibits exponential decay. Consequently, the exponential stability of both $\sz^1$ and $\frac{d}{dt}\sz^1$ is established. Therefore, by virtue of \eqref{bdd_S}, we conclude that $S(\tu)$ also decays exponentially.

Let us now consider the system \eqref{infinte}. Using Remark \ref{rem:eig-N} and the exponential decay of $S(\tu)$ in \eqref{infinte}, and proceeding analogously as above, we get
\begin{align}\label{exp_nlim}
    \|z_j(t)\|\leq Ce^{-\alpha_2 t}\|z_j(0)\|,
\end{align}
for some $C>0$, and $\alpha_2>0$ independent of $j={N_\omega}+1, {N_\omega}+2, \ldots.$ A combination of \eqref{exp_lin} and \eqref{exp_nlim} lead to 
\begin{align*}
   \|\sz(t)\|\leq Ce^{-\alpha t}\|w(0)\|,
\end{align*}
for some $C>0$ and $\alpha>0.$
\end{proof}
\section{Stabilization of Nonlinear System}\label{sec:SNlS}
In this section, we investigate the stabilization of the nonlinear system \eqref{eqn:modelled} around the steady state $\y_\infty$. Recall that, in the previous section, we introduced the transformed variables $\tx := e^{\omega t} w$ and $\tu := e^{\omega t} u$ to analyze the exponential stability of the linearized system \eqref{eqn:shift_lmodel}, which yielded a decay rate of $\omega + \alpha$, for some $\alpha > 0$. In contrast, as we turn to the nonlinear setting, we shall observe a reduction in the decay rate. Consequently, we must restrict $\omega$ to satisfy
$
\omega < \min\{\delta, \omega_0 - \epsilon\},
$
where $\omega_0$ is defined in \eqref{eqdef:omega_0} and $\epsilon > 0$ is any small number.

  For any $\omega \in (0, \min\{\delta, \omega_0-\epsilon\}),$  set $\tx:=e^{\omega t}w$ and $\tu=:e^{\omega t}u$, then $(\tx,\tu)$ satisfies
\begin{equation}  \label{eqn:shift_nlmodel}
\left\{
\begin{aligned}
& \tx_t=\eta\Delta\tx+(\omega-\beta\gamma)\tx+\int_0^t e^{-\delta(t-s)}\Delta\tx(s) \, ds +F(\tx, \y_\infty),\quad \text{ in } Q ,\\
&   \tx=\tu, \quad  \text{ on } \Sigma_1 , \quad  \frac{\partial \tx}{\partial n}= 0, \quad  \text{ on } \Sigma_2,\\&\tx(x,0)=\y_0(x)-\y_\infty(x), \, \text{ for all } x\in \Omega,
\end{aligned}
\right.
\end{equation}
where, 
\begin{align}\label{F}
    F(\tx, y_\infty)=F_1(\tx, y_\infty)+F_2(\tx, y_\infty),\end{align}
with  
\begin{equation}\label{eq:F1_F2}
\left\{
    \begin{aligned} 
	&F_1(\tx, y_\infty):= a \left( e^{-\omega\kappa t} (\tx+e^{\omega t}y_\infty)^\kappa \nabla (\tx+e^{\omega t}y_\infty) \cdot \boldsymbol{1}  - e^{\nu t}y_\infty^\kappa\nabla y_\infty\cdot \boldsymbol{1}	\right),\\
&F_2(\tx, y_\infty):= 
	\Big( \beta e^{-2\omega\kappa t} (\tx+ e^{\omega t}y_\infty)^{2\kappa+1} -\beta (1+\gamma) e^{-\omega kt} (\tx+ e^{\omega t}y_\infty)^{\kappa+1} \Big) \\&\qquad\qquad\qquad- \Big( \beta e^{\omega t}y_\infty^{2\kappa+1} -\beta (1+\gamma) e^{\omega t}y_\infty^{\kappa+1} +\frac{1}{\delta}e^{(\omega-\delta)t}\Delta y_\infty \Big).
    \end{aligned}
    \right.
\end{equation}
First, let us observe that if $\tx, \, \tu$ satisfy \eqref{eqn:shift_nlmodel}, then  $$\x=e^{-\omega t}\tx, \quad \ub=e^{-\omega t}\tu, \ \text{ for all }\ t>0,$$  satisfy \eqref{eqn:nlmodel}. Also, note that if for some feedback control, $\tu=G(\tx)$ defined by 
\begin{align}\label{fcontol3}
    \tu=G(\tx)=\lambda_1\sum_{j=1}^N\mu_j(\tx, \varphi_j)\Phi_j, \ \text{ on }\ \Sigma_1,
\end{align}
 the system \eqref{eqn:shift_nlmodel} is exponentially stable, so is \eqref{eqn:nlmodel} with control $u=G(\x)$.

Now, we define the set $$\mathcal{D}=L^\infty(0,\infty; H^1(\Omega))\cap L^2(0,\infty;H^2(\Omega))\cap L^{2(\kappa+1)}(0,\infty; L^{6(\kappa+1)}(\Omega))\cap H^1(0,\infty;L^2(\Omega)),$$ endowed with the norm $$\|\zb\|_{\D}=\|\zb\|_{L^\infty(0,\infty; H^1)}+\|\zb\|_{ L^2(0,\infty;H^2)}+\|z\|_{L^{2(\kappa+1)}(0,\infty; L^{6(\kappa+1)})}+\|\zb\|_{H^1(0,\infty;L^2)},$$ and further for any $\rho>0,$ closed ball $\D_\rho$ is defined as $$\D_\rho=\{\zb\in \D:\|\zb\|_\D\leq\rho\}.$$
Our goal is to establish the existence of a stable solution to the system \eqref{eqn:nlmodel} under the prescribed feedback control $u$. To this end, it suffices to prove the following theorem, which ensures the existence of a local solution to the transformed system \eqref{eqn:shift_nlmodel} that satisfies certain energy estimates.

\begin{Theorem}\label{th:vstab non lin}
Let us assume that $0<\omega<\min\{\delta,\omega_0-\epsilon\},$ for any $\epsilon>0,$ and $y_\infty\in  H^2$  be the steady-state solution of the system \eqref{eq:StdNSE}. Also, assume that ($\mathbf{A}1$) is satisfied and the feedback control $\tu$ is defined in \eqref{fcontol3}. Then, there exist positive constants 
$\rho_0$ and $M$ depending on $\eta$, $\delta$, $\lambda_1$, $\beta$, $\gamma$ such that, for all $0<\rho\le \rho_0$ and for all $w_0\in H^1(\Omega)$  satisfying
$$\|w_0\|_{H^1(\Omega)}\leq M\rho, \text{ and } \|\y_\infty\|_{H^2} \le M \rho,$$ the closed loop system \eqref{eqn:shift_nlmodel} admits a unique solution $w\in D_\rho$ satisfying
\begin{align}\label{est_ztilde}
    \|\tx(t)\|\leq M_1(\|w_0\|_{H^1}+\|\y_\infty\|_{H^2}), \ \text{ for all }\  t>0, 
\end{align}
where $M_1$ is a positive constant independent of $t$ and initial data.
\end{Theorem}
To prove this theorem, we first establish certain properties of the nonlinear term $F$ defined in equation \eqref{F}. Specifically, we aim to derive estimates on $F$ to ensure that it satisfies the conditions necessary for applying the Banach fixed point theorem. Therefore, before proceeding with the proof of the main theorem, we will first prove these estimates on $F$: 
\begin{Lemma}\label{self}
    Let $0<\omega<\min\{\delta,\omega_0-\epsilon\},$ for any $\epsilon>0,$ and $y_\infty\in H^2$ be the strong solution of \eqref{eq:StdNSE} and the functions $F_1$, $F_2$ be as defined in \eqref{eq:F1_F2}. Then for any $\tx\in \D$, the functions $F_1$, $F_2$ satisfy
    \begin{align*}
& (i)\ \mbox{$\displaystyle\|F_1(\tx,y_\infty)\|_{L^2(0,\infty;L^2)}\leq C_1\left(\|\tx\|^{\kappa+1}_\D + \|y_\infty\|_{ H^2} \|\tx\|^\kappa_\D + \|y_\infty\|^{\kappa}_{H^2}\|\tx\|_\D\right),$}\\
  &(ii)\ \mbox{$\displaystyle\|F_2(\tx, \y_\infty)\|_{L^2(0,\infty;L^2)}\leq C_1\left(\|\tx\|^{2\kappa+1}_\D + \|y_\infty\|^{2\kappa}_{L^{4\kappa}}\|\tx\|_\D+\|\tx\|^{\kappa+1}_\D+\|y_\infty\|^\kappa_{ H^2}\|\tx\|_\D +\|y_\infty\|_{ H^2}\right),$}
    \end{align*}
    for some constant $C_1>0$.
\end{Lemma}
\begin{proof}
   $(i)$ First, we recall \begin{align*}F_1(\tx, y_\infty)&= a e^{-\omega\kappa t} (\tx+e^{\omega t}y_\infty)^\kappa \nabla\tx \cdot \boldsymbol{1}+a\left(e^{-\omega\kappa t} (\tx+e^{\omega t}y_\infty)^\kappa \nabla e^{\omega t}y_\infty\cdot \boldsymbol{1}  - e^{\omega t}y_\infty^\kappa\nabla y_\infty\cdot \boldsymbol{1}	\right).\end{align*}
   We consider the first term and estimate using the Cauchy-Schwarz inequality as follows:
   \begin{align}\label{F11}
       \|a e^{-\omega\kappa \cdot} & (\tx(t)+e^{\omega\cdot}y_\infty)^\kappa  \nabla\tx(t)\cdot \boldsymbol{1}\|^2_{L^2(0,\infty;L^2)}\no\\&=\int_0^\infty\|a e^{-\omega\kappa t} (\tx+e^{\omega t}y_\infty)^\kappa \nabla\tx \cdot \boldsymbol{1}\|^2 dt\no\\&\leq C\Big(\int_0^\infty\|e^{-\omega\kappa t}\tx(t)^\kappa\nabla\tx(t)\cdot\boldsymbol{1}\|^2 dt+\int_0^\infty\|y_\infty^\kappa\nabla\tx(t)\cdot\boldsymbol{1}\|^2 dt\Big)\no\\
       &\leq C(\|\tx\|^{2\kappa}_{L^\infty(0,\infty;L^{3\kappa})}+\|y_\infty\|^{2\kappa}_{L^{3\kappa}})\|\tx\|^2_{L^2(0,\infty;H^2)}.
   \end{align}
For the function $\varphi(w)=(e^{-\omega t}w+y_\infty)^{\kappa}$, we have 
\begin{align}\label{aux_res2}
 (e^{-\omega t}\tx+y_\infty)^\kappa - y_\infty^\kappa&=\varphi(\tx)-\varphi(0)=\int_0^1\frac{d\varphi}{d\theta}(\theta\tx)d\theta\no\\&=\int_0^1\varphi'(\theta\tx)\tx \,d\theta=\kappa\int_0^1(\theta e^{-\omega t}\tx+y_\infty)^{\kappa-1}\tx \, d\theta.
\end{align}
Therefore, using \eqref{aux_res2}, we estimate the second term as follows:
   \begin{align}\label{F12}
       &\|a\left(e^{-\omega\kappa\cdot} (\tx+e^{\omega\cdot}y_\infty)^\kappa \nabla e^{\omega \cdot}y_\infty\cdot \boldsymbol{1}- e^{\omega\cdot}y_\infty^\kappa\nabla y_\infty\cdot \boldsymbol{1}\right)\|_{L^2(0,\infty;L^2)}\no\\
       &\qquad=a^2\int_0^\infty\|\Big((e^{-\omega t}\tx(t)+y_\infty)^\kappa-y_\infty^\kappa\Big)e^{\omega t}\nabla y_\infty\cdot\boldsymbol{1}\|^2 dt\no\\
       &\qquad=a^2\kappa^2\int_0^\infty\left\|e^{\omega(1-\kappa) t}\tx(t)\Big(\int_0^1(\theta\tx(t)+e^{\omega t}y_\infty)^{\kappa-1}d\theta\Big)\nabla y_\infty\cdot\boldsymbol{1}\right\|^2
       dt\no\\&\qquad\leq\int_0^\infty e^{2(1-\kappa)\omega t}\|\tx(t)\|^2_{L^\infty}\Big\|\Big(\int_0^1(\theta\tx(t)+e^{\omega t}y_\infty)^{\kappa-1}d\theta\Big)\Big\|^2\|\nabla y_\infty\|^2 dt
       \no\\&\qquad\leq C\|y_\infty\|^2_{H^2}\|\tx\|^{2(\kappa-1)}_{L^\infty(0, \infty; H^1)}
       \|\tx\|^2_{L^2(0, \infty; H^2)}+C\|y_\infty\|^{2\kappa}_{H^2}\|\tx\|^2_{L^2(0, \infty; H^2)}.
   \end{align}
   Adding the estimates \eqref{F11} and \eqref{F12}, one can derive the estimate in $F_1$, which concludes the proof of $(i)$.
   \vskip 0.2cm
   \noindent 
   $(ii)$ We recall
   \begin{align*}
       F_2(\tx,y_\infty)=&\beta e^{-2\omega\kappa t}\Big((\tx+ e^{\omega t}y_\infty)^{2\kappa+1}-(e^{\omega t}y_\infty)^{2\kappa+1}\Big) -\beta (1+\gamma) e^{-\omega \kappa t}\Big( (\tx+ e^{\omega t}y_\infty)^{\kappa+1} \\& \qquad -(e^{\omega t}y_\infty)^{\kappa+1}\Big)+\frac{1}{\delta}e^{(\omega-\delta)t}\Delta y_\infty.
   \end{align*}
   Then we estimate first term using \eqref{aux_res2} as follows:
   \begin{align}
       \|\beta e^{-2\omega\kappa \cdot}&\Big((\tx+ e^{\omega\cdot}y_\infty)^{2\kappa+1}-(e^{\omega \cdot}y_\infty)^{2\kappa+1}\Big)\|^2_{L^2(0,\infty; L^2)}\no\\
       & =\beta^2 \int_0^\infty \left\|e^{-2\omega \kappa t}\Big((\tx(t)+ e^{\omega t}y_\infty)^{2\kappa+1} - (e^{\omega  t} y_\infty)^{2\kappa+1}\Big)\right\|^2 dt \no\\
& =\beta^2(2\kappa+1)^2 \int_0^\infty \left\|\tx(t) e^{-2\omega\kappa t} \left(\int_0^1(\theta\tx(t)+ e^{\omega t}y_\infty)^{2\kappa} d\theta\right) \right\|^2dt\no\\ 
& \leq C \int_0^\infty \|\tx(t)^{\kappa+1}\|^2_{L^6} \|\tx(t)^\kappa\|^2_{L^3} dt + C\int_0^\infty \|\tx(t)\|^2_{L^\infty}\|y_\infty^{2\kappa}\|^2 dt\no \\
 & \leq C\|\tx\|^{2\kappa}_{L^\infty(0,\infty; L^{3\kappa})}  \|\tx\|^{2(\kappa+1)}_{L^{2(\kappa+1)}(0,\infty; L^{6(\kappa+1)})} + \|y_\infty\|^{4\kappa}_{L^{4\kappa}}\|\tx\|^2_{L^2(0,\infty;H^2)}\no\\
  & \leq C\|\tx\|^{2\kappa}_{L^\infty(0,\infty; H^1)}  \|\tx\|^{2(\kappa+1)}_{L^{2(\kappa+1)}(0,\infty; L^{6(\kappa+1)})} + \|y_\infty\|^{4\kappa}_{L^{4\kappa}}\|\tx\|^2_{L^2(0,\infty;H^2)}.
   \end{align}
   Similarly, we estimate the second term as
   \begin{align}
       \|e^{-\omega\kappa t}&\Big( (\tx+ e^{\omega t}y_\infty)^{\kappa+1} -(e^{\omega t}y_\infty)^{\kappa+1}\Big)\|^2_{L^2(0,\infty; L^2)}\no\\& =  \int_0^\infty    \left\|  e^{-\omega\kappa t}\tx(t)\left(\int_0^1(\theta\tx(t)+ e^{\nu t}y_\infty)^{\kappa} d\theta\right)   \right\|^2 dt\no \\
& \leq C \int_0^\infty \|\tx(t)^{\kappa +1}\|^2dt +\int_0^\infty \|\tx(t)y_\infty^\kappa\|^2dt\no\\
& \leq \|\tx\|^{2(\kappa+1)}_{L^{2(\kappa+1)}(0,\infty; L^{2(\kappa+1)})} + \|y_\infty\|^{2\kappa}_{L^{2\kappa}}\|\tx\|^2_{L^2(0,\infty;H^2)}\no\\
& \leq \|\tx\|^{2(\kappa+1)}_{L^{2(\kappa+1)}(0,\infty; L^{6(\kappa+1)})} + \|y_\infty\|^{2\kappa}_{L^{2\kappa}}\|\tx\|^2_{L^2(0,\infty;H^2)}.
   \end{align}
  Combining these two estimates, we conclude (ii). 
\end{proof} 
We now derive a Lipschitz-type estimate for $F$. In this context, we state the following lemma:
\begin{Lemma}\label{lip_map}
    Let $0 < \omega < \min\{\delta, \omega_0 - \epsilon\}$ for any $\epsilon > 0$, and let $y_\infty \in H^2$ denote the strong solution of \eqref{eq:StdNSE}. Consider the functions $F_1$ and $F_2$ as defined in \eqref{eq:F1_F2}.
     Then for any $\tx_1, \, \tx_2\in \D$, the functions $F_1$, $F_2$ satisfies
    \begin{align*}
	(a)  \,  \|F_1(\tx_1,y_\infty) - F_1(\tx_2,y_\infty)\|_{L^2(0,\infty; L^2)} &\leq C_2 \|\tx_1 -\tx_2\|_\D \Big( \|\tx_1 \|^\kappa_\D+\|\tx_2 \|^\kappa_\D +\|\tx_1\|_\D\|\tx_2\|^{\kappa-1}_\D \\
	& \qquad + (\|\tx_1\|^{\kappa-1}_\D + \|\tx_2\|^{\kappa-1}_\D)\|y_\infty\|_{H^2} +\|y_\infty\|^\kappa_{H^2}\\
	& \qquad + \|\tx_1\|_\D \|y_\infty\|^{\kappa-1}_{H^2} + \|y_\infty\|^{\kappa}_{H^2}   \Big), 
\end{align*}
and 
\begin{align*}
	(b)  \,	\|F_2(\tx_1,y_\infty) - F_2(\tx_2,y_\infty)\|_{L^2(0,\infty; L^2)} & \leq C_2\|\tx_1-\tx_2\|_\D\Big( \|\tx_1\|^{\kappa-1}_\D+ \|\tx_2\|^{\kappa-1}_\D + \|y_\infty\|^{\kappa-1}_{H^2} \\
	& \qquad  + \|\tx_1\|^{\kappa}_\D+ \|\tx_2\|^{\kappa}_\D + \|y_\infty\|^{\kappa}_{H^2}+\|y_\infty\|_{H^2}\Big),
\end{align*}
for some constant $C_2>0.$
\end{Lemma}
\begin{proof} (a) Note that for all $\tx_1, \tx_2\in D$, we write
\begin{align*}
   F_1(\tx_1,y_\infty) - F_1(\tx_2,y_\infty)&= a e^{-\omega \kappa t} \left(  (\tx_1+e^{\omega  t}y_\infty)^\kappa \nabla \tx_1 \cdot \boldsymbol{1} -   (\tx_2+e^{\omega  t}y_\infty)^\kappa \nabla \tx_2 \cdot \boldsymbol{1} \right) \\
 & \quad + a e^{-\omega \kappa t}\left((\tx_1+e^{\omega  t}y_\infty)^\kappa -  (\tx_2+e^{\omega  t}y_\infty)^\kappa\right)  e^{\omega t}\nabla y_\infty \cdot \boldsymbol{1}\\
 & =    a e^{-\omega \kappa t} \left( \left( (\tx_1+e^{\omega  t}y_\infty)^\kappa -   (\tx_2+e^{\omega  t}y_\infty)^\kappa \right)\nabla \tx_1 \cdot \boldsymbol{1} \right) \\
 & \quad + a e^{-\omega \kappa t} \left(  (\tx_2+e^{\omega  t}y_\infty)^\kappa \nabla (\tx_1 -\tx_2) \cdot \boldsymbol{1}  \right) \\
 & \quad + a e^{-\omega \kappa t}\left((\tx_1+e^{\omega  t}y_\infty)^\kappa -  (\tx_2+e^{\omega  t}y_\infty)^\kappa\right)  e^{\omega t}\nabla y_\infty \cdot \boldsymbol{1}.
 \end{align*}
 Let us set the function $\varphi(w)=w^{\kappa}$. Then, we have 
\begin{align}\label{aux_res1}
(\tx_1+e^{\omega  t}y_\infty)^\kappa -(\tx_2+e^{\omega  t}y_\infty)^\kappa&=\varphi(\tx_1+e^{\omega  t}y_\infty)-\varphi(\tx_2+e^{\omega t}y_\infty)\no\\&=\int_0^1\frac{d}{d\theta}\varphi\big((\theta(\tx_1+e^{\omega t}y_\infty)+(1-\theta)(\tx_2+e^{\omega t}y_\infty)\big)d\theta\no\\&=\int_0^1\varphi'((\theta(\tx_1+e^{\omega t}y_\infty)+(1-\theta)(\tx_2+e^{\omega t}y_\infty))(\tx_1-\tx_2)d\theta\no\\
&=\kappa\int_0^1(\theta\tx_1+(1-\theta)\tx_2+e^{\omega t}y_\infty)^{\kappa-1}(\tx_1-\tx_2)d\theta.
\end{align}
 Now, we estimate each term one by one, using \eqref{aux_res1} and H\"older's inequality   as follows:
 \begin{align*}
     &\|a e^{-\omega \kappa \cdot} \left( \left( (\tx_1+e^{\omega \cdot}y_\infty)^\kappa -   (\tx_2+e^{\omega\cdot}y_\infty)^\kappa \right)\nabla \tx_1 \cdot \boldsymbol{1} \right)\|_{L^2(0,\infty; L^2)}\no\\&
     =a^2\int_0^\infty \left\|e^{-\omega\kappa t}(\tx_1 - \tx_2)(t) \Big(\int_0^1\big( \theta(\tx_1(t)+(1-\theta)\tx_2(t)+ e^{\omega t}y_\infty) \big)^{\kappa-1} d\theta\Big)\nabla\tx_1(t) \cdot \boldsymbol{1}\right\|^2 dt  \no \\
     & \leq C \int_0^\infty  e^{-2\omega \kappa t} \|\tx_1(t) - \tx_2(t)\|^2_{L^6} \|\tx_1(t)^{\kappa-1} + \tx_2(t)^{\kappa-1} +e^{(\kappa-1)\omega t} y_\infty^{\kappa-1}\|^2_{L^6}\|\nabla \tx_1(t)\|^2_{L^6} dt\no\\
 & \le C \|\tx_1-\tx_2\|^2_{L^{\infty}(0,\infty; H^1)} \Big( \|\tx_1\|^{2\kappa-2}_{L^{\infty}(0,\infty; H^1)} + \|\tx_2\|^{2\kappa-2}_{L^{\infty}(0,\infty; H^1)} \no \\
 & \qquad + \|y_\infty\|^{2\kappa-2}_{H^2} \Big) \|\tx_1\|^2_{L^2(0,\infty; H^2)}.
 \end{align*}
 Again, an application of H\"older's inequality leads to
 \begin{align*}
    & \|a e^{-\omega \kappa t} \left(  (\tx_2+e^{\omega  t}y_\infty)^\kappa \nabla (\tx_1 -\tx_2) \cdot \boldsymbol{1}  \right)\|_{L^2(0,\infty; L^2)}\no\\&= \int_0^\infty \| (e^{-\omega  t}\tx_2(t)+y_\infty)^\kappa \nabla (\tx_1 -\tx_2)(t) \cdot \boldsymbol{1}\|^2 dt \no\\
& \leq C\int_0^\infty \|e^{-\omega  t}\tx_2(t)^\kappa \nabla (\tx_1 -\tx_2)(t) \cdot \boldsymbol{1}\|^2 dt + C\int_0^\infty \| y_\infty^\kappa \nabla (\tx_1 -\tx_2)(t) \cdot \boldsymbol{1}  \|^2 dt\no \\
& \leq C  \int_0^\infty  \| \tx_2(t)^\kappa\|^2_{L^3} \|\nabla (\tx_1(t) -  \tx_2(t))\|^2_{L^6} dt  +  C  \int_0^\infty  \| y_\infty^\kappa\|^2_{L^3} \|\nabla (\tx_1(t) -  \tx_2(t))\|^2_{L^6} dt\no\\
 & \leq C \left(  \|\tx_2\|^{2\kappa}_{L^{\infty}(0,\infty; L^{3\kappa})} + \| y_\infty\|^{2\kappa}_{L^{3\kappa}} \right)\|\tx_1-\tx_2\|^2_{L^2(0,\infty; H^2)}.
 \end{align*}
 Analogously, we estimate the third term
 \begin{align*}
  & \|a e^{-\omega \kappa \cdot}\left((\tx_1+e^{\omega\cdot}y_\infty)^\kappa -  (\tx_2+e^{\omega\cdot}y_\infty)^\kappa\right)  e^{\omega \cdot}\nabla y_\infty \cdot \boldsymbol{1}\|_{L^2(0,\infty; L^2)}\\&  =a^2\kappa^2\int_0^\infty \left\| e^{\omega(1-\kappa) t}(\tx_1 -\tx_2)(t) \left(\int_0^1( \theta\tx_1(t)+  (1-\theta)\tx_2(t)+e^{\omega t}y_\infty)^{\kappa-1}  d\theta\right) \nabla y_\infty \cdot \boldsymbol{1}\right\|^2 dt \no  \\
	& \leq C \int_0^\infty  e^{\omega(1-\kappa) t} \|\tx_1(t) - \tx_2(t)\|^2_{L^6} \|\tx_1(t)^{\kappa-1} + \tx_2(t)^{\kappa-1} + e^{\omega t}y_\infty^{\kappa-1}\|^2_{L^6}\|\nabla y_\infty\|^2_{L^6} dt \no\\
	& \leq C \|\tx_1-\tx_2\|^2_{L^2(0,\infty; H^1)} \Big( \|\tx_1\|^{2\kappa-2}_{L^{\infty}(0,\infty; H^1)} + \|\tx_2\|^{2\kappa-2}_{L^{\infty}(0,\infty; H^1)}+ \|y_\infty\|^{2\kappa-2}_{ H^2} \Big) \|y_\infty\|^2_{H^2}.
 \end{align*}
 Then, adding these three estimates together, we obtain (a).
 \vskip 0.2cm
 \noindent 
(b) For $\tx_1, \tx_2\in D$  and $y_\infty\in  H^2(\Omega)$, we have  \begin{align*}
    F_2(\tx_1,y_\infty) - F_2(\tx_2,y_\infty)& =  \beta e^{-2\omega \kappa t}\left( (\tx_1+ e^{\omega t}y_\infty)^{2\kappa+1} - (\tx_2+ e^{\omega t}y_\infty)^{2\kappa+1}\right) \\
    & \qquad -\beta(1+\gamma) e^{-\omega \kappa t}\left(  (\tx_1+ e^{\omega t}y_\infty)^{\kappa+1} - (\tx_2+ e^{\omega t}y_\infty)^{\kappa+1} \right).
\end{align*}
Again using \eqref{aux_res1}, the first term can be estimated as 
\begin{align*}
	\Big\| e^{-2\omega \kappa \cdot} & \Big( (\tx_1+ e^{\omega \cdot}y_\infty)^{2\kappa+1}  - (\tx_2+ e^{\omega \cdot}y_\infty)^{2\kappa+1}\Big)\Big\|^2_{L^2(0,\infty; L^2)} \\
	& =  (2\kappa+1)^2\int_0^\infty \Big\| e^{-2\omega \kappa t}(\tx_1-\tx_2)(t) \Big(\int_0^1(\theta \tx_1(t)+ (1-\theta) \tx_2(t)+ e^{\omega t}y_\infty)^{2\kappa} d\theta\Big)\Big\|^2 dt \\
	& \leq C \int_0^\infty e^{-4\omega \kappa t}\|\tx_1-\tx_2\|^2_{L^6} \big(\|\tx_1(t)\|^{4\kappa}_{L^{6\kappa}}+ \|\tx_2(t)\|^{4\kappa}_{L^{6\kappa}}+\|e^{\omega t}y_\infty\|^{4\kappa}_{L^{6\kappa}}\big) dt \\
	& \le C \|\tx_1 - \tx_2\|_{L^\infty (0,\infty; H^1)}^2 \left( 	\int_0^\infty  \|\tx_1(t)\|_{ L^{6\kappa}}^{4\kappa} dt + 	\int_0^\infty  \|\tx_2(t)\|_{ L^{6\kappa}}^{4\kappa} dt  \right) \\
	& \qquad + C \|y_\infty\|_{H^2}^2 \|\tx_1 - \tx_2\|_{L^2(0,\infty; H^1)}^2
\end{align*}
Above, we have used the embedding $H^2(\Omega)\hookrightarrow L^{6\kappa}(\Omega)$ for the values of $\kappa$ defined in \eqref{eqdef:delta}. For the case $d=3$ with $\kappa=1$ or $d=2$ with $\kappa\in\mathbb{N},$ the above estimate is enough as the Sobolev embedding $H^1(\Omega)\hookrightarrow L^{6\kappa}(\Omega)$ holds true for the above choice of $\kappa$. However, for $d=3$ with $\kappa=2,$  we use the interpolation inequality to have
\begin{align*}
	\int_0^\infty  \|\tx_i(t)\|_{ L^{6\kappa}}^{4\kappa} dt & \le C \int_0^\infty \|\tx_i(t)\|_{ L^{6(\kappa-1)}}^{2\kappa} \|\tx_i(t)\|_{ L^{6(\kappa+1)}}^{2\kappa}dt \\
	& \le C \int_0^\infty \|\tx_i(t)\|_{ L^{6(\kappa-1)}}^{2(\kappa-1)} \|\tx_i(t)\|_{ L^{6(\kappa+1)}}^{2(\kappa+1)} dt\\
	&\le C_s \|\tx_i\|_{L^\infty (0,\infty; H^1)}^{2(\kappa-1)} \|\tx_i\|_{L^{2(\kappa+1)} (0,\infty; L^{6(\kappa+1)})}^{2(\kappa+1)}, \quad i=1,2.
\end{align*} 
Using this interpalation inequality and and Sobolev embedding theorem, we can bound the first term of (b).

Similarly, we estimate the second term as follows:
\begin{align*}
    \Big\|e^{-\omega \kappa t}&\left(  (\tx_1+ e^{\omega t}y_\infty)^{\kappa+1} - (\tx_2+ e^{\omega t}y_\infty)^{\kappa+1}\right)\Big\|^2_{L^2(0,\infty; L^2)}\\& = (\kappa+1)^2 \int_0^\infty \Big\| e^{-\omega \kappa t}(\tx_1-\tx_2)(t) \Big(\int_0^1(\theta \tx_1(t)+ (1-\theta) \tx_2(t)+ e^{\omega t}y_\infty)^{\kappa} d\theta\Big)\Big\|^2 dt \\
	& \leq C \int_0^\infty e^{-2\omega \kappa t}\|\tx_1-\tx_2\|^2_{L^\infty} \big(\|\tx_1(t)\|^{2\kappa}_{L^{2\kappa}}+ \|\tx_2(t)\|^{2\kappa}_{L^{2\kappa}}+\|e^{\omega t}y_\infty\|^{2\kappa}_{L^{2\kappa}}\big)dt \\
    &\leq C\big(\|\tx_1\|^{2\kappa}_{L^\infty(0,\infty; H^1)}+\|\tx_2\|^{2\kappa}_{L^\infty(0, \infty; H^1)}+\|y_\infty\|^{2\kappa}_{H^2}\big)\|\tx_1-\tx_2\|^2_{L^2(0, \infty; H^2)}\\&\leq C\big(\|\tx_1\|^{2\kappa}_\D+\|\tx_2\|^{2\kappa}_\D+\|y_\infty\|^{2\kappa}_{H^2}\big)\|\tx_1-\tx_2\|^2_{L^2(0, \infty; H^2)},
\end{align*}
where we have used the Sobolev embedding $H^1(\Omega)\hookrightarrow L^{2\kappa}(\Omega)$, for $\kappa$  defined in \eqref{eqdef:delta}.
Therefore, by adding the above two inequalities, we conclude  the proof of (b).
\end{proof}

We are now prepared to establish the main stabilization result for the nonlinear system \eqref{eqn:modelled} in a neighborhood of its steady state solution $y_\infty$.


\begin{proof}[Proof of Theorem \ref{th:vstab non lin}]
  We use the Banach fixed point theorem to prove this theorem. For $\rho>0$ and a given $\xi\in D_\rho$, $w_0\in H^1$, let $\tx^\xi$ satisfy
    \begin{equation}\label{equ:psi_lift_gen_model}
\left\{ 
\begin{aligned}
&\tx^\xi_t = \eta\Delta\tx^\xi+(\omega-\beta\gamma)\tx^\xi+\int_0^t e^{-\delta(t-s)}\Delta\tx^\xi(s) \, ds +F(\xi,\y_\infty)  , \text{ in } Q,\\
&\tx^\xi=\tu^\xi=G(\tx^\xi), \text{ on } \Sigma_1, \, \frac{\partial\tx^\xi}{\partial n}=0, \text{ on }\Sigma_2,\\
&\tx^\xi(x,0)=\tx_0, \text{ in }\Omega,
\end{aligned}\right.
\end{equation}
where $F$ is defined in \eqref{F} and $\tu^\xi=G(\tx^\xi)=\lambda_1\sum_{j=1}^N\mu_j(\tx^\xi, \varphi_j)\Phi_j$. For any $\xi\in D$, $F(\xi,\y_\infty)$ satisfies Lemmas \ref{self} and \ref{lip_map}. Then from Theorem \ref{reg_sol}, we obtain $\tx^\xi\in D$ and obeys the estimate \eqref{reg_est}, that is,
\begin{align}\label{reg_est1}
      &\|\tx^\xi\|_{L^\infty(0,\infty; H^1)}+\|\tx^\xi\|_{L^2(0,\infty;H^2)}+\|\tx^\xi\|_{L^{2(\kappa+1)}(0,\infty; L^{6(\kappa+1)})}+\|\tx^\xi\|_{H^1(0,\infty;L^2)} \no\\ &\quad\leq C(\|w_0\|_{H^1}+\|F(\xi,\y_\infty)\|_{L^2(0,\infty;L^2)})\no\\
      &\quad \leq C\|w_0\|_{H^1} +CC_1(\|\xi\|^{\kappa+1}_\D + \|y_\infty\|_{H^2} \|\xi\|^\kappa_\D + \|y_\infty\|^{\kappa}_{H^2}\|\xi\|_\D \no\\&\qquad\qquad+\|\xi\|^{2\kappa+1}_\D +\|y_\infty\|^{2\kappa}_{L^{4\kappa}}\|\xi\|_\D +\|y_\infty\|_{H^2}),
     \end{align} 
     where $C$ and $C_1$ are the constants introduced in Theorem \ref{reg_sol} and Lemma \ref{self}, respectively. Now for 
     \begin{align}\label{rho-cndtn}
     \|w_0\|_{H^1}+\|y_\infty\|_{H^2}\leq \mathcal{K}_1\rho\  \text{ and } \ \rho\leq \mathcal{K}_1,    
     \end{align}
     where 
     \begin{align*}
         \mathcal{K}_1=\min\bigg\{\frac{1}{7C}, \frac{1}{(7CC_1)^\frac{1}{\kappa}},\ \frac{1}{(7CC_1)^\frac{1}{2\kappa}}, \frac{1}{7CC_1}\bigg\},
     \end{align*}
     we obtain
     \begin{align*}
         \|\tx^\xi\|_{L^\infty(0,\infty; H^1)}+\|\tx^\xi\|_{L^2(0,\infty;H^2)}+\|\tx^\xi\|_{L^{2(\kappa+1)}(0,\infty; L^{6(\kappa+1)})}+\|\tx^\xi\|_{H^1(0,\infty;L^2)}\\ \leq \frac{\rho}{7} +CC_1(\rho^{\kappa+1}+\mathcal{K}_1\rho^{\kappa+1}+\mathcal{K}_1^\kappa\rho^{\kappa+1}+\rho^{2\kappa+1}+\mathcal{K}_1^{2\kappa}\rho^{2\kappa+1}+\mathcal{K}_1\rho)\leq \rho.
     \end{align*}
  This concludes the proof of $\tx^\xi\in \D_\rho$.
  
  Let $w_0\in H^1$ and $\rho$ satisfy \eqref{rho-cndtn}. For all $\xi\in \D_\rho,$ we set the mapping $\mathcal{Q}:\D_\rho\to \D_\rho$ defined by $\mathcal{Q}(\xi)=\tx^\xi$, where $\tx^\xi$ is the solution of \eqref{equ:psi_lift_gen_model}. To proceed with Banach's fixed point theorem, it remains to show that the map $\mathcal{Q}$ is a contraction mapping. To this end, let $\xi_1, \xi_2\in \D_\rho$ and set $\mathcal{Z}=\tx^{\xi_1}-\tx^{\xi_2}$. Then, $\mathcal{Z}$ satisfies the following system:
\begin{equation}\label{diff_psi_nmodel}
\left\{ 
\begin{aligned}
&\mathcal{Z}_t = \eta\Delta\mathcal{Z}+(\omega-\beta\gamma)\mathcal{Z}+\int_0^t e^{-\delta(t-s)}\Delta\mathcal{Z}(s) \, ds +F(\xi_1,\y_\infty)-F(\xi_2,\y_\infty)  , \text{ in } Q,\\
&\mathcal{Z}=G(\mathcal{Z}), \text{ on } \Sigma_1, \, \frac{\partial\mathcal{Z}}{\partial n}=0, \text{ on }\Sigma_2,\\
&\mathcal{Z}=0, \text{ in }\Omega.
\end{aligned}\right.
\end{equation}
Once again using Theorem \ref{reg_sol} and  Lemma \ref{lip_map}, we obtain
\begin{align}
    \|\mathcal{Z}\|_\D&\leq C\|F(\xi_1,\y_\infty)-F(\xi_2,\y_\infty)\|_{L^2(0,\infty; L^2)}\no\\
    &\leq C\|F_1(\xi_1,\y_\infty)-F_1(\xi_2,\y_\infty)\|_{L^2(0,\infty; L^2)}+\|F_2(\xi_1,\y_\infty)-F_2(\xi_2,\y_\infty)\|_{L^2(0,\infty; L^2)}\no\\
    &\leq CC_2 \|\xi_1 -\xi_2\|_\D \Big( 2\|\xi_1 \|^\kappa_D+2\|\xi_2 \|^\kappa_\D +\|\xi_1\|_\D\|\xi_2\|^{\kappa-1}_\D+ \|\xi_1\|_\D \|y_\infty\|^{\kappa-1}_{H^2} \no\\
	& \qquad+ (\|\xi_1\|^{\kappa-1}_\D + \|\xi_2\|^{\kappa-1}_\D)(\|y_\infty\|_{H^2}+1) +2\|y_\infty\|^\kappa_{H^2} + \|y_\infty\|^{\kappa}_{H^2} +\|y_\infty\|^{\kappa-1}_{H^2}+\|y_\infty\|_{H^2}  \Big).
\end{align}
For all 
     \begin{align}\label{rho-cndtn2}
     \|y_\infty\|_{H^2}\leq \mathcal{K}_2\rho\  \text{ and } \ \rho\leq \mathcal{K}_2,    
     \end{align}
     where 
     \begin{align*}
         \mathcal{K}_2=\min\bigg\{\frac{1}{(32CC_2)^\frac{1}{\kappa}}, \frac{1}{(32CC_2)^\frac{1}{\kappa-1}}, \frac{1}{(16CC_1)^\frac{1}{\kappa-1}}, \frac{1}{(16CC_2)^\frac{1}{\kappa-2}}\bigg\},
     \end{align*}
     we obtain
     \begin{align*}
        \|\mathcal{Z}\|_\D\leq &CC_2 \|\xi_1 -\xi_2\|_\D \Big(5\rho^\kappa+\mathcal{K}_2^{\kappa-1}\rho^\kappa+2\rho^{\kappa-1}+2\mathcal{K}_2\rho^\kappa+3\mathcal{K}_2^\kappa\rho^\kappa\Big)+\frac{\rho}{16}\|\xi_1 -\xi_2\|_\D\no\\
        \leq & \frac{15}{16}\|\xi_1 -\xi_2\|_\D,
     \end{align*}
    where $C$, $C_2$ are the given constants in Theorem \ref{reg_sol} and  Lemma \ref{lip_map}, respectively. Then, choosing \begin{align}\label{rho_M}\rho_0=M=\min\{\mathcal{K}_1, \mathcal{K}_2\},\end{align} we have 
     \begin{align*}
         \|\tx^{\xi_1}-\tx^{\xi_2}\|_\D\leq \frac{15}{16}\|\xi_1 -\xi_2\|_\D.
     \end{align*}
     For all $0<\rho\leq\rho_0,$ and $M$ satisfying \eqref{rho_M}, the map $\mathcal{Q}$ defined from $\D_\rho$ to $\D_\rho$ is a contraction map and therefore using the Banach fixed point theorem $\mathcal{Q}$ has a fixed point in $\D_\rho$ and hence \eqref{eqn:shift_nlmodel} admits a solution in $\D_\rho.$ 

    Finally, for $\xi=\tx\in \D_\rho$, repeating the estimate \eqref{reg_est1}, we derive  \eqref{est_ztilde}.
   \end{proof}

\begin{Remark}
We observe that the assumption $\omega < \delta$ is no longer necessary when $y_\infty = 0$, since it was only required to estimate the last term in $F_2$ (see \eqref{eq:F1_F2}). Therefore, the exponential stability of \eqref{eqn:modelled} around the zero steady state can still be established, with a improved decay rate satisfying $0<\omega<\omega_0-\epsilon,$ for any $\epsilon>0$.
\end{Remark}
\section{Comments and Conclusion}
In this work, we have developed a finite-dimensional boundary feedback controller that stabilizes a stationary solution of the generalized Burgers-Huxley equation with memory. To the best of our knowledge, this is the first result in the literature to address stabilization for this class of systems. The feedback control is implemented on a subset of the boundary and has a linear, finite-dimensional form, explicitly described by \eqref{fcontrol}. It utilizes only $N \in \mathbb{N}$ eigenfunctions of the Laplace operator under mixed boundary conditions, making the approach computationally efficient.

The assumption (\textbf{A1}) is not crucial and can be eliminated by applying a similar method as in \cite{M17}. Additionally, it is important to highlight that the linearized system corresponding to \eqref{eqn:modelled}, when linearized around the zero solution, is exponentially stable with a decay rate of $-(\eta\lambda_1 + \beta\gamma)$ in the absence of control and with $f_s = 0$. However, under these conditions, the full nonlinear system is not stable. Introducing boundary control enhances the decay rate of the linearized system. More significantly, this same boundary control is effective in stabilizing the full nonlinear system and achieves the same decay rate when analyzed near the zero steady state.


	\medskip\noindent	\textbf{Declarations:} 

\noindent 	\textbf{Ethical Approval:}   Not applicable 


\noindent  \textbf{Conflict of interest: }On behalf of all authors, the corresponding author states that there is no conflict of interest.

\noindent 	\textbf{Authors' contributions: } All authors have contributed equally. 


\noindent 	\textbf{Availability of data and materials: } Not applicable. 

 \bibliographystyle{amsplain}
\bibliography{References} 
		
 \end{document}